\documentclass[12pt]{scrartcl}
\usepackage[english]{babel}
\usepackage[utf8x]{inputenc}
\usepackage{amsfonts}
\usepackage{amsmath}
\usepackage{amssymb}
\usepackage{amsthm}
\usepackage{mathrsfs}
\usepackage[colorlinks=false,urlcolor=blue]{hyperref}
\usepackage{listings}
\usepackage{graphicx}
\usepackage{enumerate}
\usepackage{tikz}
\usetikzlibrary{patterns}
\usepackage{pdflscape}
\usepackage{afterpage}
\usepackage{multirow}
\usepackage{tabulary}

\addtokomafont{section}{\rmfamily}
\addtokomafont{subsection}{\rmfamily}
\addtokomafont{paragraph}{\rmfamily}
\usepackage{indentfirst}

\swapnumbers
\theoremstyle{plain}
\newtheorem{satz}{Theorem}[section]
\newtheorem{lem}[satz]{Lemma}
\newtheorem{kor}[satz]{Corollary}
\newtheorem{prop}[satz]{Proposition}
\theoremstyle{definition}

\newtheorem{alg}[satz]{Algorithm}
\newcommand{\R}{\mathbb{R}}
\newcommand{\C}{\mathbb{C}}
\newcommand{\Z}{\mathbb{Z}}
\newcommand{\Ric}{\operatorname{Ric}}
\newcommand{\scal}{\operatorname{scal}}
\newcommand{\vol}{\operatorname{vol}}
\newcommand{\Sym}{\operatorname{Sym}}
\newcommand{\tr}{\operatorname{tr}}
\newcommand{\im}{\operatorname{im}}
\newcommand{\End}{\operatorname{End}}
\newcommand{\Aut}{\operatorname{Aut}}
\newcommand{\Hom}{\operatorname{Hom}}
\newcommand{\Ad}{\operatorname{Ad}}
\newcommand{\ad}{\operatorname{ad}}
\newcommand{\Cas}{\operatorname{Cas}}
\newcommand{\Sy}{\mathscr{S}}
\newcommand{\SU}{\operatorname{SU}}
\newcommand{\su}{\mathfrak{su}}
\renewcommand{\u}{\mathfrak{u}}
\newcommand{\SO}{\operatorname{SO}}
\newcommand{\so}{\mathfrak{so}}
\newcommand{\Sp}{\operatorname{Sp}}
\renewcommand{\sp}{\mathfrak{sp}}
\newcommand{\Spin}{\operatorname{Spin}}
\newcommand{\spann}{\operatorname{span}}
\newcommand{\pr}{\operatorname{pr}}
\newcommand{\m}{\mathfrak{m}}
\newcommand{\h}{\mathfrak{h}}
\newcommand{\g}{\mathfrak{g}}
\renewcommand{\k}{\mathfrak{k}}
\DeclareMathOperator*{\closedsum}{\overline{\bigoplus}}
\newcommand{\TT}{\Sy^2_{\mathrm{tt}}}
\newcommand{\e}{\mathfrak{e}}
\newcommand{\f}{\mathfrak{f}}
\renewcommand{\t}{\mathfrak{t}}
\newcommand{\A}{\mathcal{A}}
\newcommand{\LC}{\nabla}
\newcommand{\CR}{\bar\nabla}
\newcommand{\LL}{\Delta_{\mathrm{L}}}
\newcommand{\Rcr}{\bar R}
\newcommand{\StL}{\bar\Delta}
\newcommand{\PSO}{\operatorname{PSO}}
\newcommand{\LiE}{LiE}
\newcommand{\SI}{\mathcal{S}}
\newcommand{\rmE}{\mathrm{E}}
\newcommand{\rmG}{\mathrm{G}}
\newcommand{\Der}{\operatorname{Der}}
\newcommand{\EC}{E}
\newcommand{\sv}{\mathfrak{v}}
\newcommand{\sw}{\mathfrak{w}}
\newcommand{\V}{\mathcal{V}}
\newcommand{\W}{\mathcal{W}}
\newcommand{\z}{\mathfrak{z}}
\newcommand{\gl}{\mathfrak{gl}}
\newcommand{\Sl}{\mathfrak{S}}
\newcommand{\rk}{\operatorname{rk}}
\newcommand{\p}{\mathfrak{p}}
\newcommand{\Isom}{\operatorname{Iso}}
\newcommand{\isom}{\mathfrak{iso}}
\newcommand{\curvendo}{\mathcal{K}}

\title{\rmfamily The Lichnerowicz Laplacian on normal homogeneous spaces}
\author{Paul Schwahn\footnote{Institut f\"ur Geometrie und Topologie,
 Fachbereich Mathematik, Universit\"at Stuttgart, Pfaffenwaldring 57,
 70569 Stuttgart, Germany.}}
\date{\today}

\begin{document}

\maketitle

\begin{abstract}
\noindent
We give a new formula for the Lichnerowicz Laplacian on normal homogeneous spaces in terms of Casimir operators. We derive some practical estimates and apply them to the known list of non-symmetric, compact, simply connected homogeneous spaces $G/H$ with $G$ simple whose standard metric is Einstein. This yields many new examples of Einstein metrics which are stable in the Einstein--Hilbert sense, which have long been lacking in the positive scalar curvature setting.

\medskip


\medskip

\end{abstract}


%

\paragraph*{Code availability.}
The Sage code used to implement to the algorithm in Sec.~\ref{sec:algo} and with which the data in Sec.~\ref{sec:results} was generated is available on GitHub\footnote{\url{https://github.com/PSchwahn/LLBounds}}.

\paragraph*{Acknowledgments.}
The author is indebted to Prof.~G.~Weingart for the spark of inspiration that led to the exact formula for $\LL$ in terms of Casimir operators, to Prof.~E.~Lauret for introducing him to the versatile software system SageMath, and to Prof.~U.~Semmelmann for many enlightening discussions.

\pagebreak

\section{Introduction}
\label{sec:intro}

In 1961 André Lichnerowicz introduced a second order differential operator, known today as the \emph{Lichnerowicz Laplacian} $\LL$, acting on tensor fields on any Riemannian manifold $(M,g)$ \cite{Lic}. It is a generalization of the Hodge--deRham Laplacian on differential forms for which there is a Weitzenböck formula
\[d^\ast d+dd^\ast=\nabla^\ast\nabla+\curvendo(R),\]
where $\curvendo(R)$ is a fibrewise operator depending linearly on the Riemannian curvature $R$. The right hand side makes sense not only for alternating tensors fields (i.e.~differential forms) but for tensor fields of any type and is thus taken as a definition for $\LL$.

Most notably, the Lichnerowicz Laplacian occurs in the stability analysis of Einstein manifolds \cite{B87}. An \emph{Einstein metric} $g$ on $M$ is a Riemannian metric whose Ricci tensor satisfies $\Ric=\EC g$ for some constant $\EC\in\R$, called the \emph{Einstein constant} of $(M,g)$. 

Let $M$ be a compact and oriented Riemannian manifold. The \emph{Einstein--Hilbert functional}, defined as
\[\SI(g)=\int_M\scal_g\vol_g,\]
assigns to each Riemannian metric $g$ on $M$ its total scalar curvature. It is well-known that Einstein metrics on $M$ can be characterized as the critical points of $\SI$ restricted to the ILH manifold\footnote{That is, an (infinite-dimensional) manifold modeled on an \textbf{i}nverse \textbf{l}imit of \textbf{H}ilbert spaces.} of unit volume metrics.

These critical points turn out to always be saddle points. It gets more interesting once we restrict to the manifold $\Sl$ of unit volume metrics with constant scalar curvature -- then an Einstein metric can also be a local maximum, in which case it is called \emph{stable}.

Fix some Einstein metric $g$ on $M$ and consider the linearized problem. Tangent to $\Sl$ lies the space of \emph{tt-tensors} (short for \textbf{t}raceless and \textbf{t}ransverse), denoted $\TT(M)$. The transversality is merely a gauge condition in light of the diffeomorphism invariance of $\SI$. For $h\in\TT(M)$, the second variation of $\SI$ takes the form
\[\SI_g''(h,h)=-\frac{1}{2}\left(\LL h-2\EC h,h\right)_{L^2}.\]
This demonstrates a direct relation between the \emph{linear} stability of $g$ and the spectrum of $\LL$ on $\TT(M)$. It follows from the ellipticity of $\LL$ that $\SI_g''$ has finite coindex and nullity, i.e.~the maximal subspace of $\TT(M)$ on which $\SI_g''\geq0$ is finite-dimensional. Null directions for $\SI_g''$ are the \emph{infinitesimal Einstein deformations} of $g$, that is, those tt-perturbations of $g$ which preserve the Einstein condition to first order.

For the purpose of this article we drop the prefix ``linearly'' and call an Einstein metric \emph{stable} if $\LL>2\EC$ on $\TT(M)$, \emph{semistable} if $\LL\geq2\EC$ on $\TT(M)$, \emph{neutrally stable} if it is semistable and $2\EC$ is in the spectrum of $\LL$ on $\TT(M)$, and \emph{unstable} if $\LL$ has an eigenvalue $\mu<2\EC$ on $\TT(M)$.

In 1980 Koiso published a seminal article which treats the case of Riemannian symmetric spaces \cite{Koiso80}. Irreducible symmetric spaces are isotropy-irreducible, thus they carry only one invariant Riemannian metric up to homothety which, in addition, is Einstein. If $(M,g)$ is a locally symmetric space of noncompact type with no local two-dimensional factors, it is stable thanks to a curvature criterion \cite[Cor.~2.9]{Koiso80}. The case where $(M,g)$ is of compact type required a more extensive analysis which is facilitated by the key fact that $\LL$ coincides with a \emph{Casimir operator}, a representation-theoretic entity whose spectrum is straightforward to compute thanks to the theorem of Peter--Weyl, the Frobenius reciprocity theorem, and a formula of Freudenthal. This enabled Koiso to carry out the stability analysis of irreducible symmetric spaces of compact type, leaving open some gaps that were filled recently \cite{S22,SW22}.

Although the symmetric case is a particularly pleasant one, utilizing a Casimir operator is already possible once we are dealing with an $\Ad$-invariant inner product on some Lie algebra. Thus an appropriate class of spaces to extend this approach to is that of \emph{normal homogeneous spaces}, that is, homogeneous manifolds $M=G/H$ carrying an invariant Riemannian metric which is induced by an $\Ad(G)$-invariant inner product on the Lie algebra $\g$ of $G$. All normal homogeneous Einstein manifolds with $G$ simple are known: they consist of
\begin{enumerate}
 \item irreducible symmetric spaces of compact type, classified by Cartan in 1927.
 \item (non-symmetric) \emph{strongly isotropy irreducible} spaces in the sense that the identity component of the isotropy group $H$ acts irreducibly on the tangent space of $M$, classified by Wolf \cite{Wolf68} in 1968. Here $G$ is necessarily simple. These spaces were independently classified by Manturov \cite{Ma1,Ma2,Ma3} in 1961 and are also contained in a more extensive list of Krämer \cite{Kr75} from 1975.
 \item (non-symmetric) normal homogeneous Einstein manifolds with $G$ simple which are \emph{not strongly isotropy irreducible}, classified by Wang and Ziller \cite{WZ85} in 1985.
\end{enumerate}
The purpose of the present article is to find a suitable description for the Lichnerowicz Laplacian in terms of Casimir operators and initiate the stability analysis of the second and third case. We remark that if we choose $G$ connected such that $G/H$ is simply connected, then $H$ is automatically also connected. We shall thus tacitly assume these properties and speak simply of \emph{isotropy irreducible} spaces.

The third of the above classes has been investigated by E.~Lauret, J.~Lauret and C.~Will in \cite{L1,L2,LL} with regard to a weaker notion of stability, the so-called \emph{$G$-stability}. An invariant Einstein metric on a homogeneous space $G/H$ is called $G$-stable (or $G$-semistable, $G$-neutrally stable, $G$-unstable) if the respective spectral properties of the Lichnerowicz Laplacian hold on the subspace of $G$-invariant tt-tensors. In particular a $G$-unstable metric is also unstable in the classical sense. Restricted to the $G$-invariant setting, the Lichnerowicz Laplacian reduces to a term of order zero ($\frac{1}{2}\A^\ast\A$ in our notation) for which, in the naturally reductive case, a formula in terms of structural constants was developed \cite[Thm.~5.3]{L1}.

For a long time there were no known non-symmetric examples of stable Einstein metrics of \emph{positive scalar curvature} (p.s.c.). This contrasts the fact that negative sectional curvature is sufficient for stability \cite[Cor.~12.73]{B87}, or that all Einstein metrics coming from parallel spinors (which are Ricci-flat) are semistable \cite{DWW05}. On the other hand all known examples of unstable Einstein metrics so far have p.s.c. In \cite{SSW22} the stability of the p.s.c.~standard Einstein metric on the generalized Wallach space $\rmE_7/\PSO(8)$ is proved, after its $G$-stability was already shown in \cite{L2}, yielding the first known example a stable p.s.c.~Einstein metric. The result follows from the discussion of the zeroth order curvature term $\curvendo(R)$ and already utilizes Casimir operators in a crucial way.

Our aim is to treat the full second order operator $\LL$ instead. We lay out the necessary preliminaries in Sec.~\ref{sec:prelim} and develop an exact formula for $\LL$ in terms of Casimir operators in Sec.~\ref{sec:estimates}, from which two useful estimates follow. After a short digression on how to compute the relevant Casimir eigenvalues in Sec.~\ref{sec:casimir}, we give in Sec.~\ref{sec:algo} an explicit algorithm employing the new estimates in order to find lower bounds on $\LL$ on individual Fourier modes (see Sec.~\ref{sec:prelimharm} for a clarification of this term) and single out potential sources of instability. This algorithm is then applied, case-by-case, to the lists of Wolf and Wang--Ziller of compact, simply connected standard homogeneous Einstein manifolds, all of which have nonnegative sectional curvature.

In order to carry out the necessary calculations, computer assistance has been indispensable. We implemented our algorithm in the software system SageMath \cite{Sage} and heavily relied on its interface to the computer algebra package \LiE\ \cite{LiE}. Both systems are open source.

By the nature of our approach we were only able to reap the rewards of Alg.~\ref{thisalg} on a finite number of spaces. It also remains unclear in many cases whether the found potentially destabilizing Fourier modes actually contain destabilizing tt-tensors. Although our results are only partial, they produce a lot of stable examples; to be concrete,
\begin{itemize}
 \item 51 members of the isotropy irreducible families I, II, III, VII and IX (see Tables~\ref{famii} and \ref{resultsiifam}),
 \item 22 members of the isotropy reducible families XV, XVI and XVIIa (see Tables~\ref{famnii} and \ref{resultsniifam2}), the latter being the full flag manifolds $\SO(2n)/T^n$,
 \item 18 isotropy irreducible and 16 isotropy reducible exceptional spaces (see Tables~\ref{resultsiiex1}, \ref{resultsiiex2} and \ref{resultsniiex}),
\end{itemize}
totalling 107 spaces. The results are listed and discussed in detail in Sec.~\ref{sec:results}. Overall we are led to the conclusion that stable p.s.c.~Einstein metrics are not as scarce as previously believed.

\section{Preliminaries}
\label{sec:prelim}

\subsection{The Lichnerowicz Laplacian}
\label{sec:prelimll}

We begin with a compact, oriented Riemannian manifold $(M,g)$. A \emph{tensor bundle} over $M$ is a $\SO(TM)$-invariant subbundle of some tensor power of $TM$, or more abstractly, any vector bundle $VM$ associated to the frame bundle of $(M,g)$ via some representation of $\SO(n)$. On any such bundle, the \emph{standard curvature endomorphism} $\curvendo(R)$ of the Riemannian curvature $R$ is defined by
\[\curvendo(R)=\sum_{i<j}(e_i\wedge e_j)_\ast R(e_i,e_j)_\ast\in\End VM,\]
where $(e_i)$ is a local orthonormal frame of $TM$ and $A_\ast$ denotes the natural action of some $A\in\gl(T)$ on tensors as a derivation. For the sake of notational clarity we will also write $\Der_A$ instead whenever appropriate.

Note that on $TM$ itself $\curvendo(R)$ coincides with the Ricci endomorphism, i.e.
\[g(\curvendo(R)X,Y)=\Ric(X,Y).\]

Let $\LC$ denote the Levi-Civita connection of $g$, as well as the connection induced on the tensor bundle $VM$. The \emph{Lichnerowicz Laplacian} is the self-adjoint elliptic operator defined by
\[\LL=\LC^\ast\LC+\curvendo(R)\]
on sections of $VM$. It is an instance of the \emph{standard Laplace operator} on geometric vector bundles \cite{SW19}. As for any Laplace-type operator, $\LL$ has discrete spectrum accumulating only at positive infinity. The Lichnerowicz Laplacian generalizes the Hodge-deRham Laplacian in the sense that $\LL=d^\ast d+dd^\ast$ on $\Omega^p(M)$.

A tensor bundle of particular importance is $\Sym^pT^\ast M$, the bundle of covariant symmetric $p$-tensors. Its space of smooth sections will be denoted by $\Sy^p(M)$. Let $\delta$ denote the (metric) \emph{divergence} operator defined by
\[\delta:\ \Sy^{p+1}(M)\to\Sy^p(M):\quad \delta h=-\sum_ie_i\lrcorner\LC_{e_i}h.\]
Symmetric $2$-tensors $h$ that are divergence-free ($\delta h=0$, also \emph{transverse}) and trace-free ($\tr_gh=0$) are called \emph{tt-tensors}. As explained in the introduction, the space $\TT(M)$ of tt-tensors is the central stage for the stability analysis of an Einstein metric. There is the estimate
\begin{equation}
\LL\geq2\curvendo(R)\qquad\text{on }\TT(M),\label{LLqR}
\end{equation}
cf.~\cite[Prop.~6.2]{HMS16}. A sufficient criterion for stability is thus the condition $\curvendo(R)>E$ on trace-free symmetric $2$-tensors, which will serve as an important shortcut in some cases. It provides the striking advantage of only having to analyze a fibrewise term instead of a second order differential operator.

\subsection{Normal homogeneous spaces}
\label{sec:prelimhom}

Let $M=G/H$ be a reductive homogeneous space and let $\g=\h\oplus\m$ be a reductive (i.e.~$\Ad(H)$-invariant) decomposition, where $\g$ and $\h$ denote the Lie algebras of $G$ and $H$, respectively. As usual, $\m$ is identified with the tangent space of $M$ at the base point and called the \emph{isotropy representation} of $H$. There is then a one-to-one correspondence between $H$-invariant inner products on $\m$ and $G$-invariant Riemannian metrics on $M$. Without restriction we may assume $G$ to act almost effectively -- equivalently, the isotropy representation of $\h$ is faithful.

Such an invariant metric is called \emph{normal} if it is induced by the restriction $Q\big|_\m$, where $Q$ is some $\Ad(G)$-invariant inner product on $\g$. If $G$ is compact and semisimple, there is the canonical choice $Q=-B_\g$, where $B_\g$ is the (negative-definite) Killing form of $\g$. This particular metric is called the \emph{standard metric} on $M$. If $G$ is simple, then clearly every normal metric is homothetic to the standard metric.

Let $\A$ be the $G$-invariant $(2,1)$-tensor field on $M$ defined by $\A_XY=\ad_\m(X)Y=\pr_\m[X,Y]$ for $X,Y\in\m$. If $(M,g)$ is normal homogeneous, then it is also naturally reductive -- equivalently, $\A$ is totally skew-symmetric. The tensor field $\A$ can be thought of as measuring the failure of $(M,g)$ to be locally symmetric since the vanishing of $\A$ is equivalent to the third Cartan relation $[\m,\m]\subset\h$.

Let further $\CR$ denote the \emph{canonical reductive} (or \emph{Ambrose--Singer}) connection on $M$. This $G$-invariant connection has the distinctive property that it leaves every $G$-invariant tensor field parallel. In particular $\CR$ is a metric connection. It is, however, not torsion-free; notably, its torsion tensor is given by $-\A$.

For any $X\in\m$, consider the endomorphism $\A_X=\ad_\m(X)\in\so(\m)$ and extend it to tensors of valence $p$ as a derivation $\Der_{\A_X}=(\A_X)_\ast=\ad_\m^{\otimes p}(X)$. Given some tensor bundle $VM$, this defines a $\CR$-parallel bundle map
\[\A:\ VM\to T^\ast M\otimes VM:\quad v\mapsto\sum_ie^i\otimes(\A_{e_i})_\ast v\]
with metric adjoint
\[\A^\ast:\ T^\ast M\otimes VM\to VM:\quad \alpha\otimes v\mapsto-\sum_i\alpha(e_i)(\A_{e_i})_\ast v,\]
where $(e_i)$ again denotes a local orthonormal frame of $TM$. The Levi-Civita connection $\LC$ of a normal metric $g$ can then be expressed in terms of $\CR$ and $\A$ as
\begin{equation}
\LC=\CR+\frac{1}{2}\A.\label{lccr}
\end{equation}

\subsection{Casimir operators}
\label{sec:prelimcas}

Consider a real Lie algebra $\g$ equipped with an invariant inner product $Q$. Given a representation $\rho_\ast: \g\to\End V$, the \emph{Casimir operator} is a $\g$-equivariant endomorphism of $V$ defined by
\[\Cas^{\g,Q}_V:=-\sum_i\rho_\ast(e_i)^2.\]
On an irreducible module, the Casimir operator acts as multiplication with a constant as a consequence of Schur's Lemma, henceforth called the \emph{Casimir constant}. For compact semisimple $\g$, the Casimir constant of an irreducible $\g$-module $V$ with highest weight $\lambda\in\t^\ast$, where $\t\subset\g$ is a suitably chosen maximal abelian subalgebra, is given by Freudenthal's formula:
\begin{equation}
\Cas^{\g,Q}_\lambda=Q^\ast(\lambda,\lambda+2\delta_\g),\label{freudenthal}
\end{equation}
where $\delta_\g$ is the half-sum of positive roots and $Q^\ast$ is the inner product on $\t^\ast$ dual to $Q\big|_\t$.

On the other hand, if $\g$ is abelian, the Casimir constant on the weight space defined by the weight $\lambda\in\g^\ast$ is simply given by the squared length of the weight, i.e.
\begin{equation}
\Cas^{\g,Q}_\lambda=Q^\ast(\lambda,\lambda).\label{castorus}
\end{equation}

Two issues arise in practice when the Casimir constants are to be computed: first, how to find and represent the highest weights; second, how to find the appropriate inner product on the weight lattice, especially when $\g$ is not simple. Section~\ref{sec:casimir} is devoted to handling these problems.

For our purposes, it suffices to express the weights of a semisimple Lie algebra of rank $r$ in the basis of \emph{fundamental weights} $(\omega_i)_{i=1}^r$, such that each dominant integral weight $\lambda$ can be written as
\[\lambda=\sum_{i=1}^ra_i\omega_i,\qquad a_i\in\Z_{\geq0}\]
(also called \emph{coroot style notation}). We use Bourbaki's convention for the ordering of fundamental weights of a simple Lie algebra, as do \LiE\ and Sage.

Throughout what follows we will omit the superscript $Q$ in $\Cas^{\g,Q}_V$ if the inner product is clear from context. If Casimir operators of both $\g$ and a subalgebra $\h$ are present, the implied inner product on $\h$ shall be the restriction $Q\big|_\h$ unless otherwise stated. If $\g$ is compact and $Q=-B_\g$ is the standard inner product, the Casimir operator on the adjoint representation is the identity, that is
\begin{equation}
\Cas^{\g,-B_\g}_\g=1,\label{casad}
\end{equation}
which may serve as a normalization condition to find the ``right'' inner product on the weight lattice.

We remark that the Einstein condition for a standard homogeneous space is itself encoded in a Casimir operator -- namely, the standard metric on a compact homogeneous space $G/H$ is Einstein if and only if the Casimir operator of the isotropy representation $\Cas^\h_\m$ has only one eigenvalue. If this is the case, the eigenvalue is $2\EC-\frac{1}{2}$ where $\EC$ is the Einstein constant \cite[Prop.~7.89,~7.92]{B87}, cf.~\cite[Thm.~1]{WZ85}.

\subsection{Harmonic analysis on homogeneous spaces}
\label{sec:prelimharm}

Let $M=G/H$ be a homogeneous space and $\rho: H\to\Aut V$ a finite-dimensional representation. We denote with $VM=G\times_\rho V$ the associated vector bundle over $M$ with fiber $V$. Its sections are identified with $H$-equivariant $V$-valued functions on $G$, i.e.
\[\Gamma(VM)\stackrel{\cong}{\longrightarrow} C^\infty(G,V)^H:\ s\mapsto\hat s,\qquad\text{where}\quad s(xH)=[x,\hat s(x)]\in G\times_\rho V.\]
This space is an infinite-dimensional $G$-module via the \emph{left-regular representation}
\[\ell:\ G\to\Aut C^\infty(G,V)^H:\ (\ell(x)f)(y)=f(x^{-1}y),\qquad x,y\in G.\]
Every tensor bundle on $M$ can be understood as associated to a suitable tensor power of the isotropy represention $\m$ of $M$, for example $\Sym^pT^\ast M\cong G\times_\rho\Sym^p\m^\ast$.

The canonical reductive connection $\CR$, acting as covariant derivative on sections of a tensor bundle $\Gamma(VM)$, translates simply into the directional derivative on $C^\infty(G,V)^H$, i.e.
\[\widehat{\CR_Xs}=X(\hat s)=-\ell_\ast(X)\hat s,\qquad X\in\m.\]

Suppose $G$ is compact and denote with $\hat G$ the set of equivalence classes of finite-dimensional irreducible complex $G$-modules. Each such module $V_\gamma$ is (up to equivalence) uniquely determined by its highest weight $\gamma$. The set $\hat G$ is thus parametrized by the dominant integral weights of $G$, after the necessary choices have been made.

If $V$ is a unitary $H$-module, then an irreducible decomposition of the left-regular representation on sections of $VM$ is given by a consequence of the classical Peter--Weyl theorem and Frobenius reciprocity, also known as the Peter--Weyl theorem for homogeneous vector bundles \cite[Thm.~5.3.6]{Wa}. It states that
\begin{equation}
L^2(G,V)^H\cong\closedsum_{\gamma\in\hat G}V_\gamma\otimes\Hom_H(V_\gamma,V).\label{peterweyl}
\end{equation}
For each \emph{Fourier mode} $\gamma\in\hat G$ we call $\Hom_H(V_\gamma,V)=(V_\gamma^\ast\otimes V)^H$ the space of \emph{Fourier coefficients}. Given $v\in V_\gamma$ and $F\in\Hom_H(V_\gamma,V)$, the equivariant (smooth) function corresponding to $v\otimes F$ is given by $x\mapsto F(x^{-1}v)$.

Any $G$-invariant differential operator $D: \Gamma(VM)\to\Gamma(WM)$ between such vector bundles can be analyzed in the Fourier image where it consists of a discrete family $(D\big|_\gamma)_{\gamma\in\hat G}$ of linear operators
\[D\big|_\gamma:\ \Hom_H(V_\gamma,V)\longrightarrow\Hom_H(V_\gamma,W).\]
One important invariant differential operator is the standard Laplacian of the connection $\CR$, defined by
\[\StL=\CR^\ast\CR+\curvendo(\Rcr).\]
A key observation is that on normal homogeneous spaces this operator coincides with the Casimir operator of the left-regular representation \cite[Lem.~5.2]{MS10}, that is
\begin{equation}
\StL=\Cas^\g_\ell\label{stcas}
\end{equation}
(so that $\StL\big|_\gamma$ is just multiplication by the constant $\Cas^\g_\gamma$). We remark that if the underlying space is symmetric, i.e. $\A=0$, then $\CR$ coincides with the Levi-Civita connection $\LC$ and thus $\StL$ with the Lichnerowicz Laplacian $\LL$, a fact that has been of vital importance for the foundational work of Koiso \cite{Koiso80} on the stability of symmetric spaces. Our aim is to give a similarly satisfying formula for $\LL$ also in the case $\A\neq0$.

\section{Formulas and estimates for the Lichnerowicz Laplacian}
\label{sec:estimates}

Let $M=G/H$ be a homogeneous space, where $G$ is a compact Lie group, equipped with a normal Riemannian metric $g$. Let $\g$ and $\h$ denote the Lie algebras of $G$ and $H$, respectively. We begin with a description of the Lichnerowicz Laplacian on symmetric tensor fields in terms of the reductive connection $\CR$ and the tensor field $\A$.

\begin{lem}
\label{ll}
On $\Sy^p(M)$, $\LL=\StL+\A^\ast\CR+\frac{1}{2}\A^\ast\A$.
\end{lem}
\begin{proof}
By definition, $\LL=\LC^\ast\LC+\curvendo(R)$ and $\StL=\CR^\ast\CR+\curvendo(\Rcr)$. We first compare the two rough Laplacians. Noting that $\CR^\ast\A=\A^\ast\CR$ since $\A$ is $\CR$-parallel, it follows from \eqref{lccr} that
\[\LC^\ast\LC=\CR^\ast\CR+\A^\ast\CR+\frac{1}{4}\A^\ast\A.\]
Combining this with \cite[Cor.~3.2]{SSW22}, which states that $\curvendo(R)=\curvendo(\Rcr)+\frac{1}{4}\A^\ast\A$ on symmetric tensors of any valence, we obtain the desired relation.
\end{proof}

We recognize the standard Laplace operator $\StL$ of the reductive connection, which is nothing but the Casimir operator on the left-regular representation by \eqref{stcas}. Our goal is to obtain an expression of $\LL$ purely in terms of Casimir operators so that the calculation of its spectrum reduces to a representation-theoretic problem as in the symmetric case. Fortunately, this turns out to be possible. First, we need to recall an earlier result about the zeroth order term in the formula of Lemma~\ref{ll}.

\begin{lem}[\cite{SSW22}, Lem.~3.3]
\label{aacas1}
On $\m^{\otimes p}$,
\[\A^\ast\A=\pr_{\m^{\otimes p}}\Cas^\g_{\g^{\otimes p}}-\Cas^\h_{\m^{\otimes p}}-\Der_{\Cas^\h_\m}.\]
\end{lem}

Recall that $\Cas^\h_\m$ simply acts as multiplication with the constant $c=2\EC-\frac{1}{2}$ if $(M,g)$ is Einstein with Einstein constant $\EC$. Extending this as a derivation to the $p$-fold tensor power results in multiplication with $pc$. This simplifies the formula in Lemma~\ref{aacas1}.

\begin{kor}
\label{aacas2}
If $(M,g)$ is Einstein, then on $\m^{\otimes p}$,
\[\A^\ast\A=\pr_{\m^{\otimes p}}\Cas^\g_{\g^{\otimes p}}-\Cas^\h_{\m^{\otimes p}}-2p\EC+\frac{p}{2}.\]
\end{kor}

We turn now to a description of the first order differential operator $\A^\ast\CR$. By means of the inclusion $\m^{\otimes p}\subset\g^{\otimes p}$ and forgetting the $H$-invariance we can consider $C^\infty(G,\m^{\otimes p})^H$ as a subspace of the $G$-module $C^\infty(G,\g^{\otimes p})\cong C^\infty(G)\otimes\g^{\otimes p}$. On the level of Fourier coefficients this corresponds to the inclusion $(V_\gamma^\ast\otimes\m^{\otimes p})^H\subset V_\gamma^\ast\otimes\g^{\otimes p}$. Suggestively denoting the representation of $G$ on $C^\infty(G,\g^{\otimes p})$ by $\ell\otimes\Ad^{\otimes p}$, it becomes possible to write the first order term $\A^\ast\CR$ in terms of Casimir operators.

\begin{lem}
\label{1cas}
On $C^\infty(G,\m^{\otimes p})^H$,
\[\A^\ast\CR=\frac{1}{2}\Cas^\g_\ell+\frac{1}{2}\pr_{\m^{\otimes p}}(\Cas^\g_{\g^{\otimes p}}-\Cas^\g_{\ell\otimes\Ad^{\otimes p}})-\Cas^\h_{\m^{\otimes p}}.\]
\end{lem}
\begin{proof}
Let $F\in C^\infty(G,\m^{\otimes p})^H$. The $H$-invariance means precisely that $(\ell\otimes\Ad^{\otimes p})\big|_H$ acts trivially on $F$. In particular $\Cas^\h_{\ell\otimes\Ad^{\otimes p}}F=0$ and thus, if $(e_i)$ denotes an orthonormal basis of $\m$,
\begin{align*}
\Cas^\g_{\ell\otimes\Ad^{\otimes p}}F&=-\sum_i(\ell\otimes\Ad^{\otimes p})_\ast(e_i)^2F\\
&=-\sum_i\left(\ell_\ast(e_i)^2F+2\ad^{\otimes p}(e_i)\ell_\ast(e_i)F+\ad^{\otimes p}(e_i)^2F\right).
\end{align*}
Let us analyze the occurring terms separately. First,
\[-\sum_i\ell_\ast(e_i)^2F=\Cas^\g_\ell F-\Cas^\h_\ell F\]
and $\Cas^\h_\ell F=\Cas^\h_{\m^{\otimes p}}F$ by the $H$-invariance of $F$. Second, recall that for any $X\in\m$, $\CR_X$ on sections of $TM^{\otimes p}$ translates into $-\ell_\ast(X)$ on $C^\infty(G,\m^{\otimes p})^H$ and thus
\[\A^\ast\CR F=\sum_i\ad_\m^{\otimes p}(e_i)\ell_\ast(e_i)F=\pr_{\m^{\otimes p}}\sum_i\ad^{\otimes p}(e_i)\ell_\ast(e_i)F.\]
Third,
\[-\sum_i\ad^{\otimes p}(e_i)^2F=\Cas^\g_{\g^{\otimes p}}F-\Cas^\h_{\m^{\otimes p}}F.\]
After orthogonally projecting to $\m^{\otimes p}$ in the fiber, we thus obtain
\[\pr_{\m^{\otimes p}}\Cas^\g_{\ell\otimes\Ad^{\otimes p}}F=\Cas^\g_\ell F-2\A^\ast\CR F+\pr_{\m^{\otimes p}}\Cas^\g_{\g^{\otimes p}}-2\Cas^\h_{\m^{\otimes p}}\]
and the assertion follows.
\end{proof}

Combining \eqref{stcas}, Lemma~\ref{ll}, Corollary~\ref{aacas2} and Lemma~\ref{1cas}, we obtain the following final formula.

\begin{kor}
\label{LLcas}
If $(M,g)$ is Einstein, then on $C^\infty(G,\Sym^p\m)^H$,
\begin{align*}
\LL&=\frac{3}{2}\Cas^\g_\ell+\pr_{\Sym^p\m}\left(\Cas^{\g}_{\Sym^p\g}-\frac{1}{2}\Cas^\g_{\ell\otimes\Ad^{\otimes p}}\right)-\frac{3}{2}\Cas^{\h}_{\Sym^p\m}-p\EC+\frac{p}{4}.
\end{align*}
\end{kor}

This exact formula for $\LL$ is quite powerful provided the necessary representation-theoretic data is available. However, given a fixed Fourier mode $\gamma\in\hat G$, it is in general difficult to explicitly describe the two operators $\pr_{\Sym^p\m}\Cas^{\g}_{\Sym^p\g}$ and $\pr_{\Sym^p\m}\Cas^\g_{\ell\otimes\Ad^{\otimes p}}$ on the space $\Hom_H(V_\gamma,\Sym^p\m)$ of Fourier coefficients. As an example, the first of the two is treated in \cite{SSW22} on the generalized Wallach space $\rmE_7/\PSO(8)$, where it is possible to exploit additional symmetries. Indeed, for the general setting of normal homogeneous spaces this seems currently out of reach.

Nevertheless it is possible to obtain at least some estimates on $\LL\big|_\gamma$ in terms of the Fourier mode $\gamma$. We will do this in two ways. The first (\emph{crude}) estimate relies only on bounds for the fibrewise term $\A^\ast\A$, as well as the Casimir eigenvalue $\Cas^\g_\gamma$ which can be quickly computed by means of Freudenthal's formula. This has the striking advantage that the fibrewise data needs only be computed once. The second (\emph{refined}) estimate is a direct consequence of the formula in Corollary~\ref{LLcas}. It is sharper, but the problematic terms mentioned above need to be handled separately for each Fourier mode.

For the stability analysis of a given space $(M,g)$, both estimates work together effectively: the crude one rules out all but finitely many Fourier modes as candidates for instabilities, so that it remains to apply the refined one to each of the remaining Fourier modes. This synergy will be drawn on by the algorithm described in Sec.~\ref{sec:algo}.

Let $\lambda_{\mathrm{min}}[L]$ (resp. $\lambda_{\mathrm{max}}[L]$) denote the minimal (resp. maximal) eigenvalue of a self-adjoint linear operator $L$ on a finite-dimensional vector space.

\begin{satz}[Crude Estimate]
\label{crude}
For any $\gamma\in\hat G$,
\[\LL\big|_\gamma\geq\Cas^\g_\gamma+\frac{1}{2}\lambda_{\mathrm{min}}[\A^\ast\A]-\sqrt{\lambda_{\mathrm{max}}[\A^\ast\A]\cdot(\Cas^\g_\gamma-\lambda_{\mathrm{min}}[\Cas^\h_{\Sym^p\m}])}\]
on symmetric $p$-tensors.
\end{satz}
\begin{proof}
By \eqref{stcas} and Lemma~\ref{ll} we can write
\[\LL\big|_\gamma=\Cas^\g_\gamma+\A^\ast\CR\big|_\gamma+\frac{1}{2}\A^\ast\A\geq\Cas^\g_\gamma+\lambda_{\mathrm{min}}[\A^\ast\CR\big|_\gamma]+\frac{1}{2}\lambda_{\mathrm{min}}[\A^\ast\A].\]
Let now $F\in\Hom_H(V_\gamma,\Sym^p\m)$. Then
\begin{align*}
\left|\left(\A^\ast\CR F,F\right)_{L^2}\right|&=\left|\left(\CR F,\A F\right)_{L^2}\right|\leq\|\CR F\|_{L^2}\cdot\|\A F\|_{L^2},\\
\|\CR F\|^2_{L^2}&=\left(\CR^\ast\CR F,F\right)_{L^2}=\left((\Cas^\g_\gamma-\Cas^\h_{\Sym^p\m})F,F\right)_{L^2}\\
&\leq (\Cas^\g_\gamma-\lambda_{\mathrm{min}}[\Cas^\h_{\Sym^p\m}])\cdot\|F\|^2_{L^2},\\
\|\A F\|^2_{L^2}&=\left(\A^\ast\A F,F\right)_{L^2}\leq\lambda_{\mathrm{max}}[\A^\ast\A]\cdot\|F\|^2_{L^2}.
\end{align*}
Thus, the operator norm of the self-adjoint operator $\A^\ast\CR\big|_\gamma$ is bounded above by
\[\|\A^\ast\CR\big|_\gamma\|^2\leq\lambda_{\mathrm{max}}[\A^\ast\A]\cdot(\Cas^\g_\gamma-\lambda_{\mathrm{min}}[\Cas^\h_{\Sym^p\m}])\]
and together with $\lambda_{\mathrm{min}}[\A^\ast\CR\big|_\gamma]\geq-\|\A^\ast\CR\big|_\gamma\|$ the assertion follows.
\end{proof}

\begin{satz}[Refined Estimate]
\label{refined}
Suppose $(M,g)$ is Einstein and $\gamma\in\hat G$ is fixed. Let $\V=\spann\{\im F\,|\,F\in\Hom_H(V_\gamma,\Sym^p\m)\}\subset\Sym^p\m$, let $\W\subset\Sym^p\g$ denote the smallest $G$-invariant subspace containing $\V$, and likewise $\mathcal{U}\subset\Hom(V_\gamma,\Sym^p\g)$ the smallest $G$-invariant subspace containing $\Hom_H(V_\gamma,\Sym^p\m)$. Then
\begin{align*}
\LL\big|_\gamma&\geq\frac{3}{2}\Cas^\g_\gamma-\frac{1}{2}\lambda_{\mathrm{max}}\Big[\Cas^{\g}_{V_\gamma\otimes\Sym^p\g}\big|_{\mathcal{U}}\Big]\\
&\phantom{\geq}+\lambda_{\mathrm{min}}\Big[\Cas^{\g}_{\Sym^p\g}\big|_{\W}\Big]-\frac{3}{2}\lambda_{\mathrm{max}}\Big[\Cas^{\h}_{\Sym^p\m}\big|_{\mathcal{V}}\Big]-p\EC+\frac{p}{4}.
\end{align*}
on symmetric $p$-tensors.
\end{satz}
\begin{proof}
This is a direct consequence of Corollary~\ref{LLcas} if we note that $\mathcal{U}$, $\V$ and $\W$ are by construction the smallest possible subspaces on which the eigenvalues of the respective Casimir operators are of interest. Moreoever, since $\LL$ is self-adjoint, it suffices to estimate the expression $\left(\LL F,F\right)_{L^2}$ for $F\in\Hom_H(V_\gamma,\Sym^p\m)$, whence the orthogonal projections occurring in the formula of Corollary~\ref{LLcas} can be dropped.
\end{proof}

\section{Computation of Casimir eigenvalues}
\label{sec:casimir}

In the previous section the Lichnerowicz Laplacian and related quantities were expressed solely in terms of Casimir operators. For the actual computation of their eigenvalues, a few remarks are in order.

We briefly lay out our setting of interest: let $\g$ be a compact simple Lie algebra, equipped with the standard inner product $-B_\g$, and let $\h\subset\g$ be some subalgebra. In general $\h$ splits as a direct sum into
\[\h=\h_1\oplus\ldots\oplus\h_k\oplus\z,\]
where $\h_1,\ldots,\h_k$ are simple and $\z=\z(\h)$ is the central part of $\h$.

Any irreducible (complex) $\h$-module $V$ has the form
\[V=V_{\lambda_1}\otimes\ldots\otimes V_{\lambda_k}\otimes\C_{\lambda_\z},\]
where $V_{\lambda_i}$ are the highest weight modules to the weights $\lambda_i$ of $\h_i$, and $\C_{\lambda_\z}$ is the $\z$-module associated to the weight $\lambda_\z\in\z^\ast$. We collect all those into a ``highest weight'' $\lambda=(\lambda_1,\ldots,\lambda_k,\lambda_\z)$. The Casimir constant on $V$ is then simply the sum
\begin{equation}
\Cas^\h_\lambda=\Cas^{\h_1}_{\lambda_1}+\ldots+\Cas^{\h_k}_{\lambda_k}+\Cas^\z_{\lambda_\z}.\label{cascomposite}
\end{equation}

The inner product on $\h$ (and thus on its components) shall be the restriction of $-B_\g$. We discuss the simple and abelian components separately.

\paragraph{The Casimir operator on simple subalgebras.}

Let $\lambda=\sum_ia_i\omega_i$ be a weight of a simple Lie algebra $\g$ and let $\mathbf{a}=(a_1,\ldots,a_r)^\top$ be its coefficient vector. If $C_\g$ denotes the Cartan matrix of $\g$, then
\[\langle\lambda,\lambda\rangle=\mathbf{a}^\top C_\g^{-1}\mathbf{a}\]
defines an inner product which is proportional to the one induced by $-B_\g$. To find the proportionality constant, we utilize the normalization condition \eqref{casad} for the adjoint representation of $\g$. The standard Casimir constants can thus be computed with Freudenthal's formula \eqref{freudenthal} using just the inner product $\langle\cdot,\cdot\rangle$:
\[\Cas^{\g,-B_\g}_\lambda=\frac{\langle\lambda,\lambda+2\delta_\g\rangle}{\langle\lambda_{\ad},\lambda_{\ad}+2\delta_\g\rangle}=\frac{\mathbf{a}^\top C_\g^{-1}(\mathbf{a}+\mathbf{2})}{\mathbf{a}_{\ad}^\top C_\g^{-1}(\mathbf{a}_{\ad}+\mathbf{2})},\]
where $\lambda_{\ad}$ denotes the highest root of $\g$. We recall also that $\delta_\g=\omega_1+\ldots+\omega_r$, so its coefficient vector is $\mathbf{1}=(1,\ldots,1)^\top$.

Let now $\h\subset\g$ be a simple subalgebra. The Killing forms of $\g$ and $\h$ (and thus the Casimir operators of $\h$ defined by them) differ by a positive factor $b_{\g,\h}$, i.e.
\[B_\g=b_{\g,\h}B_\h,\qquad\Cas^{\h,-B_\g}=b_{\g,\h}^{-1}\Cas^{\h,-B_\h}.\]
In order to compute $b_{\g,\h}$, consider the adjoint representation of $\g$ restricted to $\h$. A simple calculation then shows that
\[\tr\Cas^{\h,-B_\g}_\g=\dim\h.\]
Thus if $\m=\m_1\oplus\ldots\oplus\m_l$ is the irreducible decomposition of the isotropy representation of $\g/\h$, we have
\[\dim\h=b_{\g,\h}^{-1}\cdot\left(\dim\h+\sum_{j=1}^l\dim\m_j\cdot\Cas^{\h,-B_\h}_{\m_j}\right),\]
from which the quantity $b_{\g,\h}$ is easily computable.

\paragraph{The Casimir operator on abelian subalgebras.}

Let now $\h\subset\g$ be abelian and let $\t_\g\subset\g$ be a maximal abelian subalgebra containing $\h$. Customarily, the inclusion $\iota: \h\hookrightarrow\t_\g\subset\g$ is characterized by a \emph{restriction matrix}; that is, a matrix $R$ representing the transposed map ${\iota^\top: \t_\g^\ast\to\h^\ast}$, where $\t_\g^\ast$ carries the basis of fundamental weights of $\g$, and $\h^\ast$ some arbitrary basis of its integral lattice.

Let us again denote with $\mathbf{a}$ the coefficient vector of a weight $\lambda\in\h^\ast$ of $\h$ with respect to the chosen basis of $\h^\ast$. The inner product on $\h^\ast$ defined by
\[\langle\lambda,\lambda\rangle=\mathbf{a}^\top(RC_\g R^\top)^{-1}\mathbf{a}\]
is then again proportional to the one coming from $-B_\g$. Repeating the trace argument above with the (complexified) isotropy representation $\m=\C_{\lambda_1}\oplus\ldots\oplus\C_{\lambda_l}$ of $\g/\h$, the Casimir constant of $\lambda$ can now by \eqref{castorus} be computed as
\[\Cas^{\h,-B_\g}_\lambda=c_\h\cdot\langle\lambda,\lambda\rangle=c_\h\cdot\mathbf{a}^\top(RC_\g R^\top)^{-1}\mathbf{a},\]
where the proportionality constant $c_\h$ is obtained by
\[\dim\h=c_\h\cdot\sum_{j=1}^l\mathbf{a}_j^\top(RC_\g R^\top)^{-1}\mathbf{a}_j.\]

\paragraph{Computation of the Einstein constant.}

Returning to the general setting of a compact simple Lie group $G$ with a closed subgroup $H$ such that the standard metric on the homogeneous space $G/H$ is Einstein, the computation of the Einstein constant $\EC$ (and the checking of the Einstein condition) is straightforward once the necessary data is assembled. Given an irreducible decomposition $\m=\m_1\oplus\ldots\oplus\m_l$ of the isotropy representation and a restriction matrix characterizing the embedding of $H$ in $G$, one computes the Casimir constants $\Cas^{\h,-B_\g}$ on each summand $\m_j$ by means of \eqref{cascomposite} and the preceding two paragraphs. We recall that the standard metric on $G/H$ is Einstein if and only if $\Cas^{\h,-B_\g}_{\m_j}$, $j=1,\ldots,l$, all act by multiplication with the same constant $c$, in which case the Einstein constant is calculated from $c=2\EC-\frac{1}{2}$.

\section{tt-tensors and Killing vector fields}
\label{sec:tt}

In Sec.~\ref{sec:estimates} we obtained general estimates for the Lichnerowicz Laplacian on $\Sy^p(M)$ if $(M,g)$ is a normal homogeneous Einstein manifold. In order to analyze the stability of $(M,g)$ we thus specialize to $p=2$. However, only the spectrum of $\LL$ on the subspace $\TT(M)$ is of relevance for the stability discussion. Thus we ought to address the issue of distinguishing the tt-tensors among $\Sy^2(M)$.

Curiously, tt-tensors are closely related to (conformal) Killing vector fields. For a compact manifold $(M^n,g)$, let
\[\delta^\ast:\ \Sy^p(M)\to\Sy^{p+1}(M):\quad \delta^\ast h=\sum_ie^i\odot\LC_{e_i}h\]
denote the formal adjoint of the divergence operator, also called the \emph{Killing operator}. In the case $p=1$, this reduces to
\[\delta^\ast\alpha=L_{\alpha^\sharp}g,\qquad\alpha\in\Omega^1(M),\]
so that $\ker\delta^\ast\big|_{\Omega^1}$ is precisely dual to the space of Killing vector fields. Taking the trace-free part we obtain a differential operator
\[\theta:\ \Omega^1(M)\to\Sy^2_0(M):\quad \theta\alpha=\delta^\ast\alpha+\frac{2}{n}\delta\alpha\cdot g\]
whose kernel is dual to the space of \emph{conformal Killing vector fields}. The relation hinted at above is now made manifest in the short exact sequence
\begin{equation}
0\longrightarrow\ker\theta\stackrel{\subset}{\longrightarrow}\Omega^1(M)\stackrel{\theta}{\longrightarrow}\Sy^2_0(M)\stackrel{P}{\longrightarrow}\TT(M)\longrightarrow0\label{sequence}
\end{equation}
(cf.~\cite[Lem.~4.1, Rem.~4.6]{S22}), where $P$ shall be the $L^2$-orthogonal projection onto $\TT(M)$. Owing to the fact that $\LL$ commutes with every arrow in \eqref{sequence}, one may obtain a similar sequence and thus a dimension formula pertaining to the eigenspaces of $\LL$ on $\Omega^1(M)$, $\Sy^2_0(M)$ and $\TT(M)$. This has indeed been utilized in the stability analysis of the irreducible symmetric spaces of compact type \cite{S22,SW22}.

Returning to the compact, Riemannian homogeneous setting $M=G/H$, we observe that every arrow of \eqref{sequence} is $G$-equivariant. Thus, introducing the linear operators
\begin{align*}
\theta\big|_\gamma&:\ \Hom_H(V_\gamma,\m)\longrightarrow\Hom_H(V_\gamma,\Sym^2_0\m),\\
\delta\big|_\gamma&:\ \Hom_H(V_\gamma,\Sym^2_0\m)\longrightarrow\Hom_H(V_\gamma,\m),
\end{align*}
(in the notation of Sec.~\ref{sec:prelimharm}) we obtain a short exact sequence
\begin{equation}
0\longrightarrow\ker\theta\big|_\gamma\longrightarrow\Hom_H(V_\gamma,\m)\longrightarrow\Hom_H(V_\gamma,\Sym^2_0\m)\longrightarrow\ker\delta\big|_\gamma\longrightarrow0
\label{seqfourier}
\end{equation}
for each Fourier mode $\gamma\in\hat G$, from which the dimension formula
\[\dim\ker\delta\big|_\gamma=\dim\Hom_H(V_\gamma,\Sym^2_0\m)-\dim\Hom_H(V_\gamma,\m)+\dim\ker\theta\big|_\gamma\]
follows.

What to make of this? The dimensions of $\Hom_H(V_\gamma,\m)$ and $\Hom_H(V_\gamma,\Sym^2_0\m)$ may easily be computed using representation theory. We now recall the following well-known fact: If $(M,g)$ is an Einstein manifold not isometric to a round sphere, then every conformal Killing vector field is Killing, i.e. $\ker\theta=\ker\delta^\ast$ \cite[Lem.~4.2]{S22}. Recall also that Killing vector fields are the infinitesimal generators of isometries. Provided $G$ acts almost effectively on $M$, a lower dimension bound on $\ker\delta^\ast$ is thus given by the inclusion
\[\g\hookrightarrow\isom(M,g)=\ker\delta^\ast:\ X\mapsto\tilde X,\quad \tilde X_p=\frac{d}{dt}\big|_{t=0}\exp(tX).p\]
mapping each Lie algebra element to the fundamental vector field generated by it. Moreover it is not hard to show that these fundamental vector fields do under left-translation in fact transform as the adjoint representation of $\g$, that is, they are of Fourier type $\lambda_{\ad}$. The corresponding Fourier coefficient in $\Hom_H(\g,\m)$ is simply the projection $\pr_\m$.

In general, $\isom(M,g)$ might be larger than $\g$, so Killing vector fields may not be confined to the Fourier mode $\lambda_{\ad}$ alone. Strikingly, in the isotropy irreducible case, a result due to Wolf tells us that this does not happen in practice:

\begin{prop}[\cite{Wolf68}, Thm.~17.1]
\label{wolfkilling}
Let $M=G/H$ be a non-Euclidean, simply connected, isotropy irreducible space with $G$ connected and effective, $K$ compact, and with a $G$-invariant Riemannian metric $g$.
\begin{itemize}
 \item If $G/H=\rmG_2/\SU(3)$, then $(M,g)$ is the round $S^6$, so $\Isom(M,g)^0=\SO(7)$.
 \item If $G/H=\Spin(7)/\rmG_2$, then $(M,g)$ is the round $S^7$, so $\Isom(M,g)^0=\SO(8)$.
 \item In every other case, $\Isom(M,g)^0=G$.
\end{itemize}
\end{prop}

Even more welcomely, Wang--Ziller extended this statement to the wider class of spaces that we are interested in.

\begin{prop}[\cite{WZ85}, Thm.~5.1]
Let $M=G/H$ be a compact, simply connected, isotropy reducible homogeneous space with $G$ compact, connected, simple and effective and a normal Einstein metric $g$. Then $\Isom(M,g)^0=G$.
\end{prop}

Having established that Killing vector fields are exclusively of Fourier type $\lambda_{\ad}$ if $(M,g)$ is not isometric to a round sphere, it follows that $\ker\theta\big|_\gamma=\ker\delta^\ast\big|_\gamma=0$ if $\gamma\neq\lambda_{\ad}$, i.e.~$\theta\big|_\gamma:\ \Hom_H(V_\gamma,\m)\to\Hom_H(V_\gamma,\Sym^2_0\m)$ is injective. We can thus formulate a corollary.

\begin{kor}
\label{dimhom}
If $M=G/H$ is a compact, simply connected homogeneous space with $G$ simple and acting almost effectively, equipped with a normal Einstein metric $g$ such that $(M,g)$ is not isometric to a round sphere, then $\ker\theta=\ker\delta^\ast\cong\g$ as a $G$-module, and $\ker\theta\big|_\gamma=0$ if $\gamma\neq\lambda_{\ad}$.
\end{kor}

When combined with \eqref{seqfourier}, we obtain a simple criterion for when a Fourier mode contains no tt-tensors, which rules them out for the stability discussion.

\begin{kor}
\label{dimhomcheck}
Under the same assumptions as in Corollary~\ref{dimhom}, a Fourier mode $\gamma\in\hat G$ contains no tt-tensors (i.e.~$\ker\delta\big|_\gamma=0$) if and only if
\[\dim\Hom_H(V_\gamma,\Sym^2_0\m)-\dim\Hom_H(V_\gamma,\m)=\begin{cases}0,&\gamma\neq\lambda_{\ad},\\-1,&\gamma=\lambda_{\ad}.\end{cases}\]
\end{kor}

This is a direct generalization of a result by Gasqui--Goldschmidt on irreducible symmetric spaces of compact type \cite[Prop.~2.40]{GG}.
 
\section{Algorithm for obtaining lower bounds on the Lichnerowicz Laplacian}
\label{sec:algo}

We lay out an overview of the algorithm used in our computations, without any explicit regard to implementation details. The necessary steps for the calculation of the Casimir constants have been subsumed in Sec.~\ref{sec:casimir}. In any of the occurring direct sum decompositions, multiplicities of irreducible summands can be disregarded since they are irrelevant for the computation. The algorithm has been implemented using SageMath (version 9.2) and its interface to the software package \LiE.

\begin{alg}[Lower Bounds for $\LL$]
\label{thisalg}
Let $M=G/H$ be a homogeneous space with $G$ compact and simple such that the standard metric $g$ is Einstein. Assume that $(M,g)$ is not isometric to a round sphere (required for Step \ref{ttstep}).
\begin{enumerate}
 \item Branch the adjoint representation on $\g$ to $\h$ to find the isotropy representation $\m$.
 \item Compute $\Cas^\h_\m$ and Einstein constant $\EC$.
 \item Decompose $\Sym^2_0\m=\bigoplus_{i\in I}\sv_i$ into $\h$-isotypes and $\Sym^2_0\g=\bigoplus_{j\in J}\sw_j$ into $\g$-isotypes.
 \item For each $j\in J$:
 \begin{enumerate}
  \item Compute $\Cas^\g_{\sw_j}$.
  \item Branch $\sw_j$ to $\h$.
 \end{enumerate}
 \item For each $i\in I$:
 \begin{enumerate}
  \item Compute $\Cas^\h_{\sv_i}$.
  \item Find $J_i=\{j\in J\,|\,\Hom_H(\sv_i,\sw_j)\neq0\}$.
  \item Find minimum/maximum of $\{\Cas^\g_{\sw_j}\,|\,j\in J_i\}$. These are lower/upper bounds for $\Cas^\g_{\Sym^2\g}$ on the smallest $G$-invariant subspace of $\Sym^2_0\g$ containing $\sv_i$.
  \item Combine to find bounds for $\A^\ast\A$ and $\curvendo(R)$ using \cite[Cor.~3.2]{SSW22} and Corollary~\ref{aacas2}.
  \item Check if $\curvendo(R)>\EC$, in which case the $\sv_i$ cannot contribute to instability.
 \end{enumerate}
 \item\label{SQ} Let $\mathcal{P}=\bigoplus\{\sv_i\,|\,\curvendo(R)\not>\EC\text{ by the above bounds}\}$. If $\mathcal{P}=0$, then $\curvendo(R)>\EC$ on $\Sym^2_0\m$ and hence $(M,g)$ is stable by \eqref{LLqR}.
 \item Combine the bounds for $\A^\ast\A$ with the crude estimate from Theorem~\ref{crude} and find $C>0$ such that $\LL\big|_\gamma>2\EC$ if $\Cas^\g_\gamma>C$.
 \item Find $\hat{G}_C=\{\gamma\in\hat G\,|\,\Cas^\g_\gamma\leq C\}$.
 \item\label{SF} For each $\gamma\in\hat{G}_C$:
 \begin{enumerate}
  \item Check whether $\Hom_H(V_\gamma,\mathcal{P})=0$. If so, then $\curvendo(R)h>\EC$ for all $h\in\Sy^2_0(M)$ of Fourier type $\gamma$. Thus $\gamma$ cannot contribute to instability by \eqref{LLqR}.
  \item Check whether $\dim\Hom_H(V_\gamma,\m)=\dim\Hom_H(V_\gamma,\Sym^2_0\m)$ ($+1$ if $\gamma=\lambda_{\ad}$). If so, then $\gamma$ contains no tt-tensors by Corollary~\ref{dimhomcheck} and thus cannot contribute to instability.\label{ttstep}
  \item Find $\V=\bigoplus\{\sv_i\,|\,\Hom_H(\sv_i,V_\gamma)\neq0\}$ (the relevant part of $\Sym^2_0\m$) and compute $\Cas^\h$ there to find $\lambda_{\mathrm{max}}[\Cas^\h_{\Sym^2\m}\big|_\V]$.
  \item Find $\W=\bigoplus\{\sw_j\,|\,\Hom_H(\sw_j,\V)\neq0\}$ (the relevant part of $\Sym^2_0\g$) and compute $\Cas^\g$ there to find $\lambda_{\mathrm{min}}[\Cas^\g_{\Sym^2\g}\big|_\W]$.
  \item Compute the tensor product of $\g$-modules $V_\gamma\otimes\W=\bigoplus_{j\in J_\gamma}\u_j$.
  \item For each $j\in J_\gamma$:
  \begin{enumerate}
   \item Compute $\Cas^\g_{\u_j}$.
   \item Branch $\u_j$ to $\h$ and check whether $\u_j^H=0$.
  \end{enumerate}
  \item Find $\lambda_{\mathrm{max}}[\Cas^{\g}_{V_\gamma\otimes\Sym^2\g}\big|_{\mathcal{U}}]=\max\{\Cas^\g_{\u_j}\,|\,\u_j^\h\neq0\}$.
  \item Combine Casimir bounds with the refined estimate from Theorem~\ref{refined} to find a lower bound for $\LL\big|_\gamma$.
 \end{enumerate}
\end{enumerate}
\end{alg}

\section{Results and discussion}
\label{sec:results}

\subsection{Setup and remarks}

We begin with listing our spaces of interest, namely
\begin{enumerate}
\item the compact, simply connected \emph{isotropy irreducible} homogeneous spaces $G/H$ which are not symmetric, as classified by Wolf \cite{Wolf68}, consisting of 10 infinite families I--X (see Table~\ref{famii}) and 33 exceptions\footnote{We do not list the space $\so(20)/\su(4)$ appearing in \cite{Wolf68} as it is a member of family X ($n=6$).} (see Tables~\ref{resultsiiex1} and~\ref{resultsiiex2}),
\item the compact, simply connected homogeneous spaces $G/H$ with $G$ simple where the standard metric is Einstein and which are \emph{isotropy reducible}, as classified by Wang and Ziller \cite{WZ85}, consisting of 9 infinite families XI--XIX (see Table~\ref{famnii}) and 22 exceptions (see Table~\ref{resultsniiex}).
\end{enumerate}

\begin{table}[p]
\centering
\small
\renewcommand{\arraystretch}{1.2}
\begin{tabulary}{\textwidth}{cccCcc}
No.&$\g$&$\h$&Condition&Defining rep.&$E$\\\hline\hline
I&$\su(\frac{n(n-1)}{2})$&$\su(n)$&$n\geq5$&$\eta_2$&$\frac{1}{4}+\frac{2}{n(n-2)}$\\\hline
II&$\su(\frac{n(n+1)}{2})$&$\su(n)$&$n\geq3$&$2\eta_1$&$\frac{1}{4}+\frac{2}{n(n+2)}$\\\hline
III&$\su(pq)$&$\su(p)\oplus\su(q)$&$2\leq p\leq q$, $p+q\neq4$&$\eta_1+\eta_1'$&$\frac{1}{4}+\frac{p^2+q^2}{2p^2q^2}$\\\hline
IV&$\sp(n)$&$\sp(1)\oplus\so(n)$&$n\geq3$&$\eta_1+\eta_1'$&$\frac{3}{8}+\frac{n+16}{8n(2n-1)}$\\\hline
V&$\so(n^2-1)$&$\su(n)$&$n\geq3$&$\eta_1+\eta_{n-1}$&$\frac{1}{4}+\frac{1}{n^2-3}$\\\hline
VI&$\so((n-1)(2n+1))$&$\sp(n)$&$n\geq3$&$\eta_2$&$\frac{1}{4}+\frac{1}{(n-1)(n+1)(2n-3)}$\\\hline
VII&$\so(2n^2+n)$&$\sp(n)$&$n\geq2$&$2\eta_1$&$\frac{1}{4}+\frac{1}{2n^2+n-2}$\\\hline
VIII&$\so(4n)$&$\sp(1)\oplus\sp(n)$&$n\geq2$&$\eta_1+\eta_1'$&$\frac{3}{8}+\frac{n+4}{8n(2n-1)}$\\\hline
IX&$\so(\frac{n(n-1)}{2})$&$\so(n)$&$n\geq7$&$\eta_2$&$\frac{1}{4}+\frac{2}{n^2-n-4}$\\\hline
X&$\so(\frac{(n-1)(n+2)}{2})$&$\so(n)$&$n\geq5$&$2\eta_1$&$\frac{1}{4}+\frac{2n}{(n-2)(n+2)(n+3)}$\\\hline
\end{tabulary}
\caption{The 10 families of \emph{isotropy irreducible} spaces.}
\label{famii}
\end{table}

\begin{table}[p]
\centering
\small
\renewcommand{\arraystretch}{1.2}
\begin{tabulary}{\textwidth}{cccCc}
No.&$\g$&$\h$&Condition&$E$\\\hline\hline
XIa&$\su(n)$&$\R^{n-1}$&$n\geq3$&$\frac{1}{4}+\frac{1}{2n}$\\\hline
XIb&$\su(kn)$&$k\su(n)\oplus(k-1)\R$&$k\geq3$, $n\geq2$&$\frac{1}{4}+\frac{1}{2n}$\\\hline
XII&$\su(l+pq)$&$\su(l)\oplus\su(p)\oplus\su(q)\oplus2\R$&$2\leq p\leq q$, $lpq=p^2+q^2+1$&$\frac{1}{4}+\frac{p^2+q^2}{2(p^2+1)(q^2+1)}$\\\hline
XIII&$\sp(kn)$&$k\sp(n)$&$k\geq3$, $n\geq1$&$\frac{1}{4}+\frac{2n+1}{4(kn+1)}$\\\hline
XIV&$\sp(3n-1)$&$\su(2n-1)\oplus\sp(n)\oplus\R$&$n\geq1$&$\frac{5}{12}$\\\hline
XV&$\so(4n^2)$&$2\sp(n)$&$n\geq2$&$\frac{1}{4}+\frac{2n+1}{2n(2n^2-1)}$\\\hline
XVI&$\so(n^2)$&$2\so(n)$&$n\geq3$&$\frac{1}{4}+\frac{n-1}{n(n^2-2)}$\\\hline
XVIIa&$\so(2n)$&$\R^n$&$n\geq3$&$\frac{1}{4}+\frac{1}{4(n-1)}$\\\hline
XVIIb&$\so(kn)$&$k\so(n)$&$k,n\geq3$&$\frac{1}{4}+\frac{n-1}{2(kn-2)}$\\\hline
XVIII&$\so(3n+2)$&$\su(n+1)\oplus\so(n)\oplus\R$&$n\geq3$&$\frac{5}{12}$\\\hline
XIX&$\so(n)$&$\h_1\oplus\ldots\oplus\h_l$&see \cite[Sec.~7]{LL}&$\frac{1}{4}+\frac{\dim\h_i}{(n-2)\dim\mathfrak{p}_i}$\\\hline
\end{tabulary}
\caption{The 9 families of \emph{isotropy reducible} spaces.}
\label{famnii}
\end{table}

Throughout what follows the spaces in question will be labeled only by pairs of Lie algebras $(\g,\h)$. There is a unique simply connected homogeneous manifold $M=G/H$ corresponding to each pair $(\g,\h)$, although $G$ and $H$ need of course not be unique. If $\g$ is classical, the embedding $\h\hookrightarrow\g$ will usually be defined via a \emph{defining representation} of $\h$ which can be of unitary, symplectic or orthogonal type and thus yields an embedding into $\g=\su(n)$, $\sp(n)$ or $\so(n)$, respectively. In this case it suffices to specify the highest weight of the defining representation, which we express in the usual basis of fundamental weights $(\eta_i)$. For semisimple $\h$, this basis will be the union of bases $(\eta_i)$, $(\eta_i')$, etc.~corresponding to each simple factor.

For isotropy reducible spaces, the definition of the embedding $\h\hookrightarrow\g$ tends to be a bit more involved. For the definitions of the families XI--XIX we refer the reader to \cite{LL} where they are discussed in proper detail.

The family XIX deserves special mention. It is defined as $\SO(\p)/H$ where $K/H$ is a (reducible) symmetric space as in \cite[Ex.~3]{WZ85}. Here $\p=\p_1\oplus\ldots\oplus\p_l$ denotes the isotropy representation of $K/H$. This construction actually gives rise to most of the standard homogeneous Einstein manifolds of the form $\SO(n)/H$ (except for $\Spin(8)/\rmG_2$). In particular it already completely covers the families XV--XVIII. In order to divide up the spaces in question more evenly we impose the same contraints as given in \cite[Sec.~7]{LL} and collect everything not listed among the families XV--XVIII or the exceptions into our ``band of outcasts'' XIX.

Tables~\ref{famii} and \ref{famnii} also list the Einstein constants of the families I--XIX. They can be derived \emph{a priori} using the results of Wang and Ziller. For each isotropy irreducible space $G/H$ the Casimir constant of its isotropy representation is in their notation given as $c=E(\chi)/\alpha_G$, where $\alpha_G$ and $E(\chi)$ are listed in the tables on \cite[pp.~583, 588]{WZ85}. The Einstein constant is then $\EC=\frac{1}{4}+\frac{c}{2}$. Similarly the Einstein constants of the isotropy reducible families are derived in \cite{LL}.


Tables~\ref{resultsiiex1}, \ref{resultsiiex2}, \ref{resultsiifam}, \ref{resultsniiex}, \ref{resultsniifam1}, \ref{resultsniifam2}, \ref{resultsniifam3} are to be read as follows. The column \emph{Potential instabilities} lists all Fourier modes $\gamma\in\hat G$ for which Alg.~\ref{thisalg} does \emph{not} yield the estimate $\LL\big|_\gamma>2\EC$. If the weaker estimate $\LL\big|_\gamma\geq 2\EC$ holds, the Fourier mode $\gamma$ is printed in \textcolor{blue}{blue}. Each $\gamma$ will be expressed in the basis $(\omega_i)$ of fundamental weights of $\g$.

The abbreviations in the column \emph{Notes} will stand for the following:
\begin{itemize}
 \item \textbf{SC}: stable by Step \ref{SQ} of Alg.~\ref{thisalg}. That is, $\curvendo(R)>E$ is fulfilled, which is sufficient for stability by \eqref{LLqR}.
 \item \textbf{SF}: stable by Step \ref{SF} of Alg.~\ref{thisalg}, i.e.~after applying the new estimates (Theorem~\ref{crude} and \ref{refined}).
 \item \textbf{SF}$_0$: semistable by Step \ref{SF} of Alg.~\ref{thisalg}.
\end{itemize}

We remark that the spaces $\so(7)/\g_2$ and $\g_2/\su(3)$ from Tables~\ref{resultsiiex1} and \ref{resultsiiex2} are round spheres and thus already known to be stable. Moreover the Berger space $\sp(2)/\su(2)$ can also be written as $\so(5)/\so(3)$ (the highest weight of defining representation is then $4\eta_1$) and the Fourier mode $\omega_2$ (expressed as weight of $\sp(2)$) is known to be destabilizing \cite[Sec.~5]{SWW22}. The stability of the space $\e_7/\so(8)$ in Table~\ref{resultsniiex} was shown recently \cite{SSW22}.

\subsection{Discussion of results}

We begin with discussing the isotropy irreducible case. Tables~\ref{resultsiiex1} and \ref{resultsiiex2} show the obtained results for the exceptional spaces. In some cases (mostly those with $\g$ of type $\rmE$), the curvature estimate $\curvendo(R)>E$ is already sufficient to prove stability. This phenomenon persists for the isotropy reducible spaces, cf.~Table~\ref{resultsniiex}. We note that the spaces
\[\frac{\so(8)}{\g_2},\ \frac{\so(26)}{\sp(1)\oplus\sp(5)\oplus\so(6)},\ \frac{\f_4}{\so(8)},\ \frac{\e_6}{\so(8)\oplus\R^2},\ \frac{\e_7}{3\su(2)\oplus\so(8)}\]
were shown to be $G$-unstable in \cite{L2,LL} -- this corresponds to the Fourier mode listed as ``$0$''. Remarkably, combining our analysis with the $G$-stability results of \cite{L2,LL}, which rule out the ``$0$'' mode, leads to (semi-)stability for the spaces
\[\frac{\e_6}{3\su(2)},\ \frac{\e_7}{7\su(2)},\ \frac{\e_8}{2\su(5)},\ \frac{\e_8}{2\so(8)}.\]
Notable is also the space $\frac{\e_7}{3\su(2)\oplus\so(8)}$, shown to be $G$-unstable in \cite{LL} (with a $G$-coindex of $2$). According to our analysis, ``$0$'' is the only potential instability -- so the coindex coincides with the $G$-coindex here.

Considering the isotropy irreducible families (Table~\ref{resultsiifam}), we observe varying behavior with respect to stability. Within the scope of our computational capacity, we could show stability for the following spaces:
\begin{align*}
&\text{from I:}&\su(\frac{n(n-1)}{2})&/\su(n),&8&\leq n\leq11,\\
&\text{from II:}&\su(\frac{n(n+1)}{2})&/\su(n),&6&\leq n\leq8,\\
&\text{from III:}&\su(pq)&/(\su(p)\oplus\su(q)),&p=3\text{ and }12&\leq q\leq 16,\\
&&&&p\geq 4\text{ and }24&\leq pq\leq 49,\\
&\text{from VII:}&\so(2n^2+n)&/\sp(n),&3&\leq n\leq13,\\
&\text{from IX:}&\so(\frac{n(n-1)}{2})&/\so(n),&7&\leq n\leq27.
\end{align*}

We turn next to the isotropy reducible families (Tables~\ref{resultsniifam1}, \ref{resultsniifam2} and \ref{resultsniifam3}). These were extensively studied in \cite{L1,L2,LL}, where the $G$-instability of XI--XIV, XVIIb, XVIII and XIX was already proved. This leaves open the cases XVI ($G$-stable) and XVIIa ($G$-neutrally stable) as well as the family XV where the $G$-stability type is still unknown. We managed to show stability for the following examples:
\begin{align*}
&\text{from XV:}&\so(4n^2)&/(\sp(n)\oplus\sp(n)),&3&\leq n\leq9,\\
&\text{from XVI:}&\so(n^2)&/(\so(n)\oplus\so(n)),&4&\leq n\leq16,
\end{align*}
as well as semistability of $\so(2n)/\R^n$ from XVIIa if $n=6,7$, which follows from the $G$-semistability result of \cite{LL}.

Judging from the examples seen here and comparing with \cite{L1,L2,LL}, it seems as if $G$-(semi)-stability implies (semi)-stability in the wider sense \emph{for isotropy reducible spaces}. Of particular interest as potential counterexamples to this hypothesis are the families XV, XVI and XVIIa as well as the space $\frac{\e_6}{\su(2)\oplus\so(6)}$, which is known to be $G$-neutrally stable \cite{LL}, but where we cannot yet exclude the possibility of non-invariant destabilizing directions.

The effectiveness of Alg.~\ref{thisalg}, Step \ref{ttstep} in ruling out instabilities is underwhelming in light of how useful the same method is on symmetric spaces \cite{S22,SW22}, only eliminating the Fourier modes $\omega_2$ on $\frac{\so(8)}{\su(3)}$ (from the family V), $\omega_2$ on $\frac{\so(18)}{\so(4)\oplus\sp(3)}$ (from the family XIX, defined by the isotropy representation of the symmetric space $S^4\times\frac{\SU(6)}{\Sp(3)}$), and $\omega_2$ on $\frac{\so(26)}{\sp(1)\oplus\sp(5)\oplus\so(6)}$ (exceptional). It indicates however that the absence of tt-tensors in a given Fourier mode is quite a rare phenomenon.

In general we observe a trend towards stability in the infinite families as the rank increases. Some cases (I, II, III, VII, IX, XV, XVI, XVIIa) seem to become and stay stable at some point. For others (IV, V, VI, VIII, X) there seem to be some Fourier modes that always harbor potential instabilities.

\subsection{Outlook}

In order to decide the stability of the cases with remaining potential instabilities, it would be sufficient to compute the Lichnerowicz Laplacian seperately on each of the potentially destabilizing Fourier modes. In particular this requires the problems mentioned in Sec.~\ref{sec:estimates} to be overcome. Moreover, in order to tackle the stability analysis on the countable families in their entirety, a systematic approach to (at least) estimating the Lichnerowicz Laplacian on each family would be needed.

Another matter entirely is the question of how our approach may be generalized even further to non-normal metrics -- say, metrics that are ``almost normal'' in the sense that they can be written as
\[g=\alpha_1Q\big|_{\m_1}+\alpha_2Q\big|_{\m_2},\qquad\m=\m_1\oplus\m_2,\qquad\alpha_1,\alpha_2>0\]
for some $\Ad(G)$-invariant inner product $Q$ on $\g$. This is the case for metrics in the canonical variation of a homogeneous fibration $H/K\hookrightarrow G/K\twoheadrightarrow G/H$ with normal fiber and base. Special symmetries of the form
\[[\m_1,\m_1]\subset\k,\qquad[\m_2,\m_2]\subset\k\oplus\m_1\]
that may simplify computations are available if we suppose that fiber and base are symmetric, but also for Kähler--Einstein metrics on generalized flag manifolds with $b_2(M)=1$ and two isotropy summands.

In the homogeneous fibration setting there is a natural first choice of tt-tensors to investigate, namely those of the form $h=\beta_1Q\big|_{\m_1}+\beta_2Q\big|_{\m_2}$ with $\beta_1,\beta_2\in\R$ chosen such that $\tr_gh=0$. These are an instance of \emph{Killing tensors} (that is, they are annihilated by the Killing operator $\delta^\ast$). Trace-free Killing tensors have the advantage that they realize the equality in \eqref{LLqR}, that is
\begin{equation}
\Delta_Lh=2\curvendo(R)h,\label{eqkilling}
\end{equation}and the fibrewise term $\curvendo(R)$ is in general easier to handle. Moreover, Killing tensors have often been sources of instability -- see \cite{cwang} for tensors of the particular form above on fiber bundles, and \cite{SWW22} for an exploitation of \eqref{eqkilling} to show instability of the Berger space $\SO(5)/\SO(3)$.

We aim to return to all of these issues in future work.

\clearpage

\begin{table}[p]
\centering
\footnotesize
\renewcommand{\arraystretch}{1.2}
\begin{tabular}{cccccc}\hline
$\g$&$\h$&Defining rep.&$E$&Potential instabilities&Notes\\\hline\hline
$\su(16)$&$\so(10)$&$\eta_4$&$\frac{11}{32}$&$\omega_1+\omega_{15}$&\\\hline
$\su(27)$&$\e_6$&$\eta_1$&$\frac{11}{36}$&\textcolor{blue}{$\omega_1+\omega_{26}$}&\textbf{SF}$_0$\\\hline
$\so(7)$&$\g_2$&$\eta_1$&$\frac{9}{20}$&--&\textbf{SC}\\\hline
$\so(133)$&$\e_7$&$\eta_3$&$\frac{135}{524}$&--&\textbf{SF}\\\hline
$\sp(2)$&$\su(2)$&$3\eta_1$&$\frac{9}{20}$&$\omega_2,2\omega_2,2\omega_1+\omega_2,4\omega_1$&\\\hline
$\sp(7)$&$\sp(3)$&$\eta_3$&$\frac{29}{80}$&$\omega_2$&\\\hline
$\sp(10)$&$\su(6)$&$\eta_3$&$\frac{15}{44}$&$\omega_2$&\\\hline
$\sp(16)$&$\so(12)$&$\eta_5$&$\frac{43}{136}$&$\omega_2$&\\\hline
$\sp(28)$&$\e_7$&$\eta_7$&$\frac{17}{58}$&--&\textbf{SF}\\\hline
$\so(14)$&$\g_2$&$\eta_2$&$\frac{1}{3}$&$2\omega_1$&\\\hline
$\so(16)$&$\so(9)$&$\eta_4$&$\frac{23}{56}$&$2\omega_1,\omega_4$&\\\hline
$\so(26)$&$\f_4$&$\eta_4$&$\frac{1}{3}$&$\omega_1,2\omega_1$&\\\hline
$\so(42)$&$\sp(4)$&$\eta_4$&$\frac{19}{70}$&--&\textbf{SF}\\\hline
$\so(52)$&$\f_4$&$\eta_1$&$\frac{27}{100}$&--&\textbf{SF}\\\hline
$\so(70)$&$\su(8)$&$\eta_4$&$\frac{179}{680}$&--&\textbf{SF}\\\hline
$\so(78)$&$\e_6$&$\eta_2$&$\frac{5}{19}$&--&\textbf{SF}\\\hline
$\so(128)$&$\so(16)$&$\eta_7$&$\frac{173}{672}$&--&\textbf{SF}\\\hline
$\so(248)$&$\e_8$&$\eta_8$&$\frac{125}{492}$&--&\textbf{SF}\\\hline
\end{tabular}
\caption[caption]{Results for the isotropy irreducible exceptions with $\g$ classical.}
\label{resultsiiex1}
\end{table}

\begin{table}[p]
\centering
\footnotesize
\renewcommand{\arraystretch}{1.2}
\begin{tabular}{ccccc}\hline
$\g$&$\h$&$E$&Potential instabilities&Notes\\\hline\hline
$\e_6$&$\su(3)$&$\frac{11}{36}$&--&\textbf{SC}\\\hline
$\e_6$&$3\su(3)$&$\frac{5}{12}$&$\omega_2,\omega_1+\omega_6$&\\\hline
$\e_6$&$\g_2$&$\frac{25}{72}$&--&\textbf{SC}\\\hline
$\e_6$&$\su(3)\oplus\g_2$&$\frac{19}{48}$&$\omega_2,\omega_1+\omega_6$&\\\hline
$\e_7$&$\su(3)$&$\frac{71}{252}$&--&\textbf{SC}\\\hline
$\e_7$&$\su(6)\oplus\su(3)$&$\frac{5}{12}$&$\omega_1,\omega_6$&\\\hline
$\e_7$&$\g_2\oplus\sp(3)$&$\frac{7}{18}$&--&\textbf{SF}\\\hline
$\e_7$&$\su(2)\oplus\f_4$&$\frac{47}{108}$&$\omega_6$&\\\hline
$\e_8$&$\su(9)$&$\frac{5}{12}$&--&\textbf{SC}\\\hline
$\e_8$&$\e_6\oplus\su(3)$&$\frac{5}{12}$&--&\textbf{SF}\\\hline
$\e_8$&$\g_2\oplus\f_4$&$\frac{23}{60}$&--&\textbf{SF}\\\hline
$\f_4$&$2\su(3)$&$\frac{5}{12}$&$\omega_4,\omega_1,\omega_3,2\omega_4$&\\\hline
$\f_4$&$\su(2)\oplus\g_2$&$\frac{29}{72}$&$\omega_4,2\omega_4,\omega_1+\omega_4$&\\\hline
$\g_2$&$\su(2)$&$\frac{43}{112}$&$2\omega_1,\omega_1+\omega_2,2\omega_2$&\\\hline
$\g_2$&$\su(3)$&$\frac{5}{12}$&--&\textbf{SC}\\\hline
\end{tabular}
\caption[caption]{Results for the isotropy irreducible exceptions with $\g$ exceptional. The embedding $\h\subset\g$ is always characterized by $\h$ being a maximal subalgebra.}
\label{resultsiiex2}
\end{table}

\begin{table}[p]
\centering
\footnotesize
\renewcommand{\arraystretch}{1.3}
\begin{tabulary}{\textwidth}{cCcCc}\hline
Family&Param.&$r=\rk\g$&Potential instabilities&Notes\\\hline\hline
\multirow{2}{*}{I}&$n=5,6,7$&$9,14,20$&$\omega_1+\omega_r$&\\\cline{2-5}
&$n=8\ldots11$&$27,35,44,54$&--&\textbf{SF}\\\hline\hline
\multirow{3}{*}{II}&$n=3$&$5$&$\omega_1+\omega_5$, $\omega_1+\omega_2$, $\omega_2+\omega_4$, $3\omega_1$& *\\\cline{2-5}
&$n=4,5$&$9$&$\omega_1+\omega_r$&\\\cline{2-5}
&$n=6,7,8$&$20,27,35$&--&\textbf{SF}\\\hline\hline
\multirow{6}{*}{III}&$(p,q)=(2,3)$&$5$&$\omega_1+\omega_5$, $\omega_2+\omega_4$, $2\omega_1+\omega_4$, $2\omega_1+2\omega_5$& *\\\cline{2-5}
&$(p,q)=(2,4)$&$7$&$\omega_1+\omega_7$, $\omega_4$, $\omega_1+\omega_3$, $\omega_2+\omega_6$, $2\omega_2$& *\\\cline{2-5}
&$(p,q)=(3,3)$&$8$&$\omega_1+\omega_8$, $\omega_3$, $\omega_1+\omega_2$& *\\\cline{2-5}
&$(p,q)=(2,5)$&$9$&$\omega_1+\omega_9$, $\omega_2+\omega_8$&\\\cline{2-5}
&$12\leq pq\leq 22$; or $(2,q)$ with $12\leq q\leq 20$; or $(3,q)$ with $8\leq q\leq 11$&$pq-1$&$\omega_1+\omega_r$&\\\cline{2-5}
&$(3,q)$ with $12\leq q\leq 16$; or $(p,q)$ with $p\geq 4$ and $24\leq pq\leq 49$&$pq-1$&--&\textbf{SF}\\\hline\hline
\multirow{5}{*}{IV}&$n=3$&$3$&$\omega_2$, $\omega_1+\omega_3$, $2\omega_2$, $2\omega_1+\omega_2$, $4\omega_1$, $\omega_1+\omega_2+\omega_3$, $3\omega_1+\omega_3$&\\\cline{2-5}
&$n=4$&$4$&$\omega_2$, $\omega_1+\omega_3$, $2\omega_2$, $2\omega_1+\omega_2$, $4\omega_1$, \textcolor{blue}{$\omega_2+\omega_4$, $2\omega_1+\omega_4$}&\\\cline{2-5}
&$n=5$&$5$&$\omega_2$, $\omega_1+\omega_3$, $2\omega_2$, $4\omega_1$&\\\cline{2-5}
&$n=6$&$6$&$\omega_2$, $2\omega_2$&\\\cline{2-5}
&$n=7\ldots100$&$n$&$\omega_2$&\\\hline\hline
\multirow{3}{*}{V}&$n=3$&$4$&$\omega_1$, $\omega_1+\omega_3$, $2\omega_1$, $\omega_1+\omega_2$, $\omega_1+\omega_3+\omega_4$, $2\omega_1+\omega_3$, \textcolor{blue}{$2\omega_2$}&**\\\cline{2-5}
&$n=4$&$7$&$\omega_1$, $2\omega_1$, \textcolor{blue}{$\omega_2$}&\\\cline{2-5}
&$n=5\ldots18$&$\lfloor\frac{n^2-1}{2}\rfloor$&$\omega_1$&\\\hline\hline
\multirow{2}{*}{VI}&$n=3$&$7$&$\omega_1$, $2\omega_1$&\\\cline{2-5}
&$n=4\ldots11$&$\lfloor\frac{(n-1)(2n+1)}{2}\rfloor$&$\omega_1$&\\\hline\hline
\multirow{2}{*}{VII}&$n=2$&$5$&$2\omega_1$, $\omega_3$, $\omega_4+\omega_5$, $\omega_1+\omega_2$&\\\cline{2-5}
&$n=3\ldots13$&$\lfloor\frac{2n^2+n}{2}\rfloor$&--&\textbf{SF}\\\hline\hline
\multirow{4}{*}{VIII}&$n=2$&$4$&$2\omega_4$, $\omega_2+\omega_4$, $2\omega_2$&\\\cline{2-5}
&$n=3$&$6$&$2\omega_1$, $\omega_4$, $2\omega_2$&\\\cline{2-5}
&$n=4$&$8$&$2\omega_1$, $\omega_1+\omega_7$, $\omega_4$&\\\cline{2-5}
&$n=5\ldots 75$&$2n$&$2\omega_1$&\\\hline\hline
IX&$n=7\ldots27$&$\lfloor\frac{n(n-1)}{4}\rfloor$&--&\textbf{SF}\\\hline\hline
X&$n=5\ldots22$&$\lfloor\frac{(n-1)(n+2)}{4}\rfloor$&$\omega_1$&\\\hline
\end{tabulary}
\caption[caption]{Some results for the isotropy irreducible families I--X.\\\hspace{\textwidth}\footnotesize * To obtain all potential instabilities from the listed ones, take closure under the duality automorphism of A$_r$ which sends $\omega_k\mapsto\omega_{r+1-k}$.\\\hspace{\textwidth} ** To obtain all potential instabilities from the listed ones, take closure under the automorphisms of D$_4$ which permute $\omega_1$, $\omega_3$ and $\omega_4$.}
\label{resultsiifam}
\end{table}

\begin{table}[p]
\centering
\footnotesize
\renewcommand{\arraystretch}{1.5}
\begin{tabulary}{\textwidth}{cCCcCc}\hline
$\g$&$\h$&Embedding&$E$&Potential instabilities&Notes\\\hline\hline
$\so(8)$&$\g_2$&$\g_2\stackrel{\eta_1}{\hookrightarrow}\so(7)\subset\so(8)$&$\frac{5}{12}$&$0$, $\omega_1$, $\omega_1+\omega_3$, $2\omega_1$, $\omega_1+\omega_2$, \textcolor{blue}{$\omega_1+\omega_3+\omega_4$, $2\omega_1+\omega_3$}&*\\\hline
$\so(26)$&$\sp(1)\oplus\sp(5)$ $\oplus\,\so(6)$&$\sp(1)\oplus\sp(5)\stackrel{\eta_1+\eta_1'}{\hookrightarrow}\so(20)$&$\frac{29}{80}$&$0$, $2\omega_1$&\\\hline
$\f_4$&$\so(8)$&$\so(8)\subset\so(9)\stackrel{\text{max.}}{\subset}\f_4$&$\frac{4}{9}$&$0$, $\omega_4$, $\omega_3$, $2\omega_4$, \textcolor{blue}{$\omega_3+\omega_4$}&\\\hline
$\e_6$&$3\su(2)$&$\su(2)\stackrel{2\eta_1}{\hookrightarrow}\su(3)$, $3\su(3)\stackrel{\text{max.}}{\subset}\e_6$&$\frac{5}{16}$&$0$&\cite{LL}$\Rightarrow$\textbf{SF}$_0$\\\hline
$\e_6$&$\su(2)\oplus\so(6)$&$\so(6)\subset\su(6)$, $\su(2)\oplus\su(6)\stackrel{\text{max.}}{\subset}\e_6$&$\frac{3}{8}$&$0$, $\omega_2$&\\\hline
$\e_6$&$\so(8)\oplus\R^2$&$\so(8)\oplus\R\subset\so(10)$, $\so(10)\oplus\R\stackrel{\text{max.}}{\subset}\e_6$&$\frac{5}{12}$&$0$, $\omega_2$, $\omega_1+\omega_6$, \textcolor{blue}{$\omega_4$}&\\\hline
$\e_6$&$\R^6$&max. torus&$\frac{7}{24}$&--&\textbf{SC}\\\hline
$\e_7$&$7\su(2)$&$3\so(4)\subset\su(12)$, $\su(12)\oplus\su(2)\stackrel{\text{max.}}{\subset}\e_7$&$\frac{1}{3}$&$0$&\cite{LL}$\Rightarrow$\textbf{SF}\\\hline
$\e_7$&$\so(8)$&$\so(8)\subset\su(8)\stackrel{\text{max.}}{\subset}\e_7$&$\frac{13}{36}$&--&\textbf{SC}\\\hline
$\e_7$&$3\su(2)$ $\oplus\,\so(8)$&$\so(8)\oplus\so(4)\subset\su(12)$, $\su(12)\oplus\su(2)\stackrel{\text{max.}}{\subset}\e_7$&$\frac{7}{18}$&$0$&\\\hline
$\e_7$&$\R^7$&max. torus&$\frac{5}{18}$&--&\textbf{SC}\\\hline
$\e_8$&$8\su(2)$&$4\so(4)\subset\so(16)\stackrel{\text{max.}}{\subset}\e_8$&$\frac{3}{10}$&--&\textbf{SC}\\\hline
$\e_8$&$4\su(3)$&$3\su(3)\stackrel{\text{max.}}{\subset}\e_6$, $\e_6\oplus\su(2)\stackrel{\text{max.}}{\subset}\e_8$&$\frac{19}{60}$&--&\textbf{SC}\\\hline
$\e_8$&$4\su(2)$&$\su(2)\stackrel{2\eta_1}{\hookrightarrow}\su(3)$, $4\su(3)\subset\e_8$ as above&$\frac{11}{40}$&--&\textbf{SC}\\\hline
$\e_8$&$2\su(3)$&$2\su(3)\stackrel{\eta_1+\eta_1'}{\hookrightarrow}\su(9)$, $\su(9)\stackrel{\text{max.}}{\subset}\e_8$&$\frac{17}{60}$&--&\textbf{SC}\\\hline
$\e_8$&$2\su(5)$&max. subalgebra&$\frac{7}{20}$&0&\cite{LL}$\Rightarrow$\textbf{SF}\\\hline
$\e_8$&$\so(9)$&$\so(9)\subset\su(9)\stackrel{\text{max.}}{\subset}\e_8$&$\frac{13}{40}$&--&\textbf{SC}\\\hline
$\e_8$&$\so(9)$&$\so(9)\stackrel{\eta_4}{\hookrightarrow}\so(16)\stackrel{\text{max.}}{\subset}\e_8$&$\frac{13}{40}$&--&\textbf{SC}\\\hline
$\e_8$&$2\so(8)$&$2\so(8)\subset\so(16)\stackrel{\text{max.}}{\subset}\e_8$&$\frac{11}{30}$&0&\cite{L2}$\Rightarrow$\textbf{SF}\\\hline
$\e_8$&$\so(5)$&max. subalgebra&$\frac{13}{48}$&--&\textbf{SC}\\\hline
$\e_8$&$2\sp(2)$&$2\sp(2)\stackrel{\eta_1+\eta_1'}{\hookrightarrow}\so(16)\stackrel{\text{max.}}{\subset}\e_8$&$\frac{7}{24}$&--&\textbf{SC}\\\hline
$\e_8$&$\R^8$&max. torus&$\frac{4}{15}$&--&\textbf{SC}\\\hline
\end{tabulary}
\caption[caption]{Results for the isotropy reducible exceptions.\\\hspace{\textwidth}\footnotesize * To obtain all potential instabilities/IED from the listed ones, take closure under the automorphisms of D$_4$ which permute $\omega_1$, $\omega_3$ and $\omega_4$.}
\label{resultsniiex}
\end{table}

\begin{table}[p]
\centering
\footnotesize
\renewcommand{\arraystretch}{1.2}
\begin{tabulary}{\textwidth}{cCcCc}\hline
Family&Param.&$r=\rk\g$&Potential instabilities&Notes\\\hline\hline
\multirow{4}{*}{XIa}&$n=3$&$2$&$0$, $\omega_1+\omega_2$, $3\omega_1$, $2\omega_1+2\omega_2$&*\\\cline{2-5}
&$n=4$&$3$&$0$, $\omega_1+\omega_3$, $2\omega_2$, $2\omega_1+\omega_2$, $2\omega_1+2\omega_3$&*\\\cline{2-5}
&$n=5$&$4$&$0$, $\omega_1+\omega_4$, $\omega_2+\omega_3$, \textcolor{blue}{$2\omega_1+\omega_3$}&*\\\cline{2-5}
&$n=6,7,8,9$&$5,6,7,8$&$0$, $\omega_1+\omega_r$&\\\hline\hline
\multirow{7}{*}{XIb}&$(k,n)=(3,2)$&$5$&$0$, $\omega_1+\omega_5$, $\omega_2+\omega_4$, $2\omega_3$, $2\omega_1+\omega_4$, $2\omega_1+2\omega_5$&*\\\cline{2-5}
&$(k,n)=(4,2)$&$7$&$0$, $\omega_1+\omega_7$, \textcolor{blue}{$\omega_2+\omega_6$}&\\\cline{2-5}
&$(k,n)=(3,3)$&$8$&$0$, $\omega_1+\omega_8$, $\omega_2+\omega_7$, $2\omega_1+\omega_7$, \textcolor{blue}{$2\omega_1+2\omega_8$}&*\\\cline{2-5}
&$(k,n)=(5,2)$&$9$&$0$, $\omega_1+\omega_9$&\\\cline{2-5}
&$(k,n)=(3,4)$&$11$&$0$, $\omega_1+\omega_{11}$, $\omega_2+\omega_{10}$, \textcolor{blue}{$2\omega_1+\omega_{10}$}&*\\\cline{2-5}
&$(k,n)=(4,3),(6,2)$&$11$&$0$, $\omega_1+\omega_{11}$&\\\hline\hline
XII&$(p,q,l)=(2,5,3)$&$12$&$0$, $\omega_1+\omega_{12}$&\\\hline\hline
\multirow{15}{*}{XIII}&$(k,n)=(3,1)$&$3$&$0$, $\omega_2$, $2\omega_1$, $\omega_1+\omega_3$, $2\omega_2$, $2\omega_1+\omega_2$&\\\cline{2-5}
&$(k,n)=(4,1)$&$4$&$0$, $\omega_2$, $2\omega_1$, $\omega_4$, $\omega_1+\omega_3$, $2\omega_2$&\\\cline{2-5}
&$(k,n)=(5,1)$&$5$&$0$, $\omega_2$, $2\omega_1$, $\omega_4$, $\omega_1+\omega_3$&\\\cline{2-5}
&$(k,n)=(6,1)$&$6$&$0$, $\omega_2$, $2\omega_1$, $\omega_4$&\\\cline{2-5}
&$(k,n)=(3,2)$&$6$&$0$, $\omega_2$, $2\omega_1$, $\omega_4$, $\omega_1+\omega_3$, $\omega_6$, $2\omega_2$&\\\cline{2-5}
&$(k,n)=(4,2)$&$8$&$0$, $\omega_2$, $2\omega_1$, $\omega_4$&\\\cline{2-5}
&$(k,n)=(3,3)$&$9$&$0$, $\omega_2$, $2\omega_1$, $\omega_4$, $\omega_1+\omega_3$, $2\omega_2$&\\\cline{2-5}
&$(k,n)=(10,1)$&$10$&$0$, $\omega_2$, \textcolor{blue}{$2\omega_1$}&\\\cline{2-5}
&$(k,n)=(11,1)$&$11$&$0$, $\omega_2$&\\\cline{2-5}
&$(k,n)=(3,4)$&$12$&$0$, $\omega_2$, $2\omega_1$, $\omega_4$, $\omega_1+\omega_3$&\\\cline{2-5}
&$(k,n)=(7,1)$, $(8,1)$, $(9,1)$, $(5,2)$, $(6,2)$, $(4,3)$, $(7,2)$, $(5,3)$, $(3,5)$, $(8,2)$, $(4,4)$, $(6,3)$, $(3,6)$&$kn$&$0$, $\omega_2$, $2\omega_1$&\\\hline\hline
\multirow{14}{*}{XIV}&$n=1$&$2$&$0$, $\omega_2$, $2\omega_1$, $2\omega_2$, $2\omega_1+\omega_2$, $4\omega_1$, $3\omega_2$, $2\omega_1+2\omega_2$&\\\cline{2-5}
&$n=2$&$5$&$0$, $\omega_2$, $2\omega_1$, $\omega_4$, $\omega_1+\omega_3$, $2\omega_2$, $2\omega_1+\omega_2$, $\omega_1+\omega_5$, $4\omega_1$, \textcolor{blue}{$\omega_2+\omega_4$, $2\omega_1+\omega_4$}&\\\cline{2-5}
&$n=3$&$8$&$0$, $\omega_2$, $2\omega_1$, $\omega_4$, $\omega_1+\omega_3$, $2\omega_2$, $2\omega_1+\omega_2$, $4\omega_1$, \textcolor{blue}{$\omega_6$}&\\\cline{2-5}
&$n=4$&$11$&$0$, $\omega_2$, $2\omega_1$, $\omega_4$, $\omega_1+\omega_3$, $2\omega_2$, $2\omega_1+\omega_2$, \textcolor{blue}{$4\omega_1$}&\\\cline{2-5}
&$n=5$&$14$&$0$, $\omega_2$, $2\omega_1$, $\omega_4$, $\omega_1+\omega_3$, \textcolor{blue}{$2\omega_2$, $2\omega_1+\omega_2$}&\\\cline{2-5}
&$n=6$&$17$&$0$, $\omega_2$, $2\omega_1$, $\omega_4$, \textcolor{blue}{$\omega_1+\omega_3$}&\\\cline{2-5}
&$n=7$&$20$&$0$, $\omega_2$, $2\omega_1$, \textcolor{blue}{$\omega_4$}\\\cline{2-5}
&$n=8,9,10$&$23,26,29$&$0$, $\omega_2$, $2\omega_1$&\\\hline
\end{tabulary}
\caption[caption]{Some results for the isotropy reducible families XI--XIV.\\\hspace{\textwidth}\footnotesize * To obtain all potential instabilities from the listed ones, take closure under the duality automorphism of A$_r$ which sends $\omega_k\mapsto\omega_{r+1-k}$.}
\label{resultsniifam1}
\end{table}

\begin{table}[p]
\centering
\footnotesize
\renewcommand{\arraystretch}{1.5}
\begin{tabulary}{\textwidth}{cCcCc}\hline
Family&Param.&$r=\rk\g$&Potential instabilities&Notes\\\hline\hline
\multirow{2}{*}{XV}&$n=2$&$8$&$0$, $2\omega_1$&\\\cline{2-5}
&$n=3\ldots9$&$2n^2$&--&\textbf{SF}\\\hline\hline
\multirow{2}{*}{XVI}&$n=3$&$4$&$0$, $\omega_1$, $2\omega_1$, $\omega_3$, $2\omega_4$&\\\cline{2-5}
&$n=4\ldots16$&$\lfloor\frac{n^2}{2}\rfloor$&--&\textbf{SF}\\\hline\hline
\multirow{3}{*}{XVIIa}&$n=4$&$4$&$0$, $\omega_2$, $2\omega_1$, $2\omega_3$, $2\omega_4$&\\\cline{2-5}
&$n=5$&$5$&$0$, $\omega_2$, \textcolor{blue}{$2\omega_1$}&\\\cline{2-5}
&$n=6,7$&$6,7$&$0$&\cite{LL}$\Rightarrow$\textbf{SF}$_0$\\\hline\hline
\multirow{10}{*}{XVIIb}&$(k,n)=(3,3)$&$4$&$0$, $\omega_2$, $2\omega_1$, $\omega_3$, $2\omega_4$, $\omega_1+\omega_2$&\\\cline{2-5}
&$(k,n)=(3,4)$&$6$&$0$, $\omega_2$, $2\omega_1$, $\omega_4$&\\\cline{2-5}
&$(k,n)=(6,3)$&$9$&$0$, \textcolor{blue}{$\omega_2$}&\\\cline{2-5}
&$(k,n)=(6,4)$&$12$&$0$, $\omega_2$, \textcolor{blue}{$2\omega_1$}&\\\cline{2-5}
&$(k,n)=(4,3)$, $(5,3)$, $(4,4)$, $(5,4)$; or $(k,5)$ with $3\leq k\leq 6$; or $(k,n)$ with $n\geq6$ and $kn\leq40$&$\lfloor\frac{kn}{2}\rfloor$&$0$, $\omega_2$, $2\omega_1$&\\\cline{2-5}
&$(k,n)=(7,3)$, $(8,3)$, $(9,3)$, $(7,4)$, $(10,3)$, $(7,5)$&$\lfloor\frac{kn}{2}\rfloor$&$0$&\\\hline\hline
\multirow{8}{*}{XVIII}&$n=3$&$5$&$0$, $\omega_1$, $\omega_2$, $2\omega_1$, $\omega_3$, $\omega_4$, $\omega_1+\omega_2$, $2\omega_5$, $\omega_1+\omega_3$, $2\omega_2$, \textcolor{blue}{$\omega_1+\omega_4$}&\\\cline{2-5}
&$n=4$&$7$&$0$, $\omega_2$, $2\omega_1$, $\omega_4$, $\omega_1+\omega_3$, \textcolor{blue}{$2\omega_2$}\\\cline{2-5}
&$n=5$&$8$&$0$, $\omega_2$, $2\omega_1$, $\omega_4$, \textcolor{blue}{$\omega_1+\omega_3$}&\\\cline{2-5}
&$n=6$&$10$&$0$, $\omega_2$, $2\omega_1$, $\omega_4$&\\\cline{2-5}
&$n=7\ldots 19$&$\lfloor\frac{3n}{2}\rfloor+1$&$0$, $\omega_2$, $2\omega_1$&\\\hline
\end{tabulary}
\caption[caption]{Some results for the isotropy reducible families XV--XVIII.}
\label{resultsniifam2}
\end{table}

\begin{table}[p]
\centering
\footnotesize
\renewcommand{\arraystretch}{1.5}
\begin{tabulary}{\textwidth}{cCccc}\hline
Family&$K/H=K_1/H_1\times\ldots\times K_l/H_l$&$r=\rk\g$&Potential instabilities\\\hline\hline
\multirow{12}{*}{XIX}&$\frac{\SU(3)^2}{\SO(3)^2}$, $S^3\times\SO(5)$, $S^4\times\frac{\SU(6)}{\Sp(3)}$&$5,6,9$&$0$, $\omega_1$, $2\omega_1$\\\cline{2-4}
&$S^3\times\SU(3)$&$5$&$0$, $\omega_1$, $\omega_2$, $2\omega_1$, $\omega_3$, \textcolor{blue}{$\omega_1+\omega_2$}\\\cline{2-4}
&$(S^3)^2\times\SU(3)$&$7$&$0$, $\omega_1$, $\omega_2$, $2\omega_1$\\\cline{2-4}
&$\SU(3)^2$, $(S^3)^3\times\SU(3)$&$8$&$0$, $\omega_1$, $\omega_2$, \textcolor{blue}{$2\omega_1$}\\\cline{2-4}
&$S^3\times \rmG_2$&$8$&$0$, \textcolor{blue}{$2\omega_1$}\\\cline{2-4}
&$S^3\times\SU(4)$, $\SU(3)\times\SO(5)$&$9$&$0$, $\omega_1$, \textcolor{blue}{$\omega_2$}\\\cline{2-4}
&$(S^4)^2\times\frac{\SU(6)}{\Sp(3)}$&$11$&$0$, $\omega_1$, $2\omega_1$, \textcolor{blue}{$\omega_2$}\\\cline{2-4}
&$(S^3)^k\times\SO(5)$ with $3\leq k\leq5$; or $(S^3)^k\times \rmG_2$ with $2\leq k\leq 4$; or $(S^3)^k\times\SO(5)^2$ with $0\leq k\leq 2$; or $S^3\times\SO(7)$, $S^3\times\Sp(3)$, $\SO(5)\times \rmG_2$&$\lfloor\frac{\dim K/H}{2}\rfloor$&$0$\\\cline{2-4}
&all other with $\dim K/H\leq 26$&$\lfloor\frac{\dim K/H}{2}\rfloor$&$0$, $\omega_1$\\\hline
\end{tabulary}
\caption[caption]{Some results for the isotropy reducible family XIX.\\\hspace{\textwidth}\footnotesize
$K/H$ denotes the symmetric space used in the construction of $M=\SO(\p)/H$. Concerning the individual symmetric factors $K_i/H_i$ we insist that Lie groups are presented as $\frac{H_i\times H_i}{H_i}$ and spheres as $S^k=\frac{\SO(k+1)}{\SO(k)}$.}
\label{resultsniifam3}
\end{table}


\clearpage
\begin{thebibliography}{uf}
\bibitem{B87}
 \textsc{A.~L.~Besse:}
 \textit{Einstein manifolds},
 \textrm{Ergebnisse der Mathematik und ihrer Grenzgebiete 3rd Series \textbf{10}, Springer-Verlag, Berlin (1987).}

\bibitem{LiE}
 \textsc{A.~M.~Cohen, M.~van~Leeuwen \& B.~Lisser:}
 \textit{LiE, a Computer algebra package for Lie group computations}
 Version~2.2.2, (2022), \url{http://wwwmathlabo.univ-poitiers.fr/~maavl/LiE/}.
 
\bibitem{DWW05}
 \textsc{X.~Dai, X.~Wang \& G.~Wei:}
 \textit{On the stability of Riemannian manifold with parallel spinors},
 \textrm{Invent. Math.~\textbf{161} (1), (2005), 151--176.}

\bibitem{GG}
 \textsc{J.~Gasqui, H.~Goldschmidt:}
 \textit{Radon transforms and the rigidity of the Grassmannians},
 \textrm{Annals of Mathematics Studies~\textbf{156}, Princeton University Press, Princeton, NJ (2004).}

\bibitem{HMS16}
 \textsc{H.~Heil, A.~Moroianu \& U.~Semmelmann:}
 \textit{Killing and conformal Killing tensors},
 \textrm{J. Geom. Phys.~\textbf{106}, (2016), 383--400.}

\bibitem{Koiso80}
 \textsc{N.~Koiso:}
 \textit{Rigidity and stability of Einstein metrics -- The case of compact symmetric spaces},
 \textrm{Osaka J. Math.~\textbf{17} (1), (1980), 51--73.}

\bibitem{Kr75}
 \textsc{M.~Krämer:}
 \textit{Eine Klassifikation bestimmter Untergruppen kompakter zusammenhängender Liegruppen},
 \textrm{Communications in Algebra~\textbf{3} (8), (1975), 691--737.}

\bibitem{L1}
 \textrm{J.~Lauret:}
 \textit{On the stability of homogeneous Einstein manifolds},
 \textrm{Asian J. Math.~\textbf{26} (4), (2022), 555--584.}

\bibitem{L2}
 \textsc{J.~Lauret, C.~Will:}
 \textit{On the stability of homogeneous Einstein manifolds II},
 \textrm{J. Lond. Math. Soc.~\textbf{106}, (2022), 3638--3669}

\bibitem{LL}
 \textsc{E.~Lauret, J.~Lauret:}
 \textit{The stability of standard homogeneous Einstein manifolds},
 \textrm{Math. Z.~\textbf{303}, (2023), 16.}
 
\bibitem{Lic}
 \textsc{A.~Lichnerowicz:}
 \textit{Propagateurs et commutateurs en relativité générale},
 \textrm{Publications Mathématiques de l'IHÉS~\textbf{10}, (1961), 5--56.}

\bibitem{Ma1}
 \textsc{O.~V.~Manturov:}
 \textit{Homogeneous asymmetric Riemannian with an irreducible group of motions},
 \textrm{Dokl. Akad. Nauk. SSSR \textbf{141}, (1961), 792--795.}
 
\bibitem{Ma2}
 \textsc{O.~V.~Manturov:}
 \textit{Riemannian spaces with orthogonal and symplectic groups of motions and an irreducible group of rotations},
 \textrm{Dokl. Akad. Nauk. SSSR \textbf{141}, (1961), 1034--1037.}
 
\bibitem{Ma3}
 \textsc{O.~V.~Manturov:}
 \textit{Homogeneous Riemannian manifolds with irreducible isotropy group},
 \textrm{Trudy Sem. Vector. Tenzor. Anal.~\textbf{13}, (1966), 68--145.}

\bibitem{MS10}
 \textsc{A.~Moroianu, U.~Semmelmann:}
 \textit{The Hermitian Laplace operator on nearly Kähler manifolds},
 \textrm{Comm. Math. Phys.~\textbf{294}, (2010), 251--272.}

\bibitem{Sage}
 \textsc{The Sage Developers:}
 \textit{SageMath, the Sage Mathematics Software System}
 Version~9.2, (2020), \url{https://www.sagemath.org}.
 
\bibitem{S22}
 \textsc{P.~Schwahn:}
 \textit{Stability of Einstein metrics on symmetric spaces of compact type},
 \textrm{Ann. Global Anal. Geom.~\textbf{61} (2), (2022), 333--357.}

\bibitem{SSW22}
 \textsc{P.~Schwahn, U.~Semmelmann \& G.~Weingart:}
 \textit{Stability of the Non--Symmetric Space $\rmE_7/\PSO(8)$},
 \textrm{Adv. Math.~\textbf{432}, (2023), 109268.}
 
\bibitem{SWW22}
 \textsc{U.~Semmelmann, C.~Wang \& M.~Wang},\quad
 \textit{Linear Instability of Sasaki Einstein and nearly parallel $\rmG_2$ manifolds},\quad
 \textrm{Internat. J. Math.~\textbf{33} (6), (2022), 2250042.}

\bibitem{SW19}
 \textsc{U.~Semmelmann, G.~Weingart:}
 \textit{The standard Laplace operator},
 \textrm{Manuscripta Math.~\textbf{158} (1--2), (2019), 273--293.} 

\bibitem{SW22}
 \textsc{U.~Semmelmann, G.~Weingart:}
 \textit{Stability of Compact Symmetric Spaces},
 \textrm{J. Geom. Anal.~\textbf{32} (4), (2022), 137.}

\bibitem{Wa}
 \textsc{N.~R.~Wallach:}
 \textit{Harmonic Analysis on Homogeneous Spaces},
 \textrm{Marcel Dekker, Inc., 1973.}

\bibitem{cwang}
 \textsc{C. Wang, Y.~K.~Wang},
 \textit{Stability of Einstein metrics on fiber bundles},
 \textrm{J. Geom. Anal.~\textbf{31} (1), (2021), 490--515.}

\bibitem{WZ85}
 \textsc{M.~Wang, W.~Ziller:} 
 \textit{On normal homogeneous Einstein manifolds},
 \textrm{Ann. Sci. École Norm. Sup.~\textbf{18} (4), (1985), 563--633.}

\bibitem{Wolf68}
 \textsc{J.~A.~Wolf:}
 \textit{The Geometry and Structure of Isotropy Irreducible Homogeneous Spaces},
 \textrm{Acta Mathematica~\textbf{120}, (1968), 59--148;}
 \textit{Correction}, \textrm{Acta Mathematica~\textbf{152}, (1984), 141--142.}
\end{thebibliography}
\end{document}